\documentclass{amsart}
\usepackage{amscd,amssymb}
\usepackage[shortlabels]{enumitem}
\usepackage{color}
\usepackage[all]{xy}
\renewcommand{\mod}{\operatorname{mod}\nolimits}

\DeclareMathOperator{\ad}{ad}

\newcommand{\rad}{\operatorname{rad}\nolimits}

\newcommand{\mult}{\operatorname{mult}\nolimits}
\newcommand{\Hom}{\operatorname{Hom}\nolimits}
\newcommand{\End}{\operatorname{End}\nolimits}

\newcommand{\Ker}{\operatorname{Ker}\nolimits}

\newcommand{\Tr}{\operatorname{Tr}\nolimits}
\newcommand{\Ext}{\operatorname{Ext}\nolimits}

\newcommand{\frako}{\mathfrak{o}}
\newcommand{\frakt}{\mathfrak{t}}
\newcommand{\G}{\Gamma}
\renewcommand{\L}{\Lambda}
\newcommand{\Z}{{\mathbb Z}}

\newcommand{\C}{{\mathcal C}}
\newcommand{\D}{{\mathcal D}}

\newtheorem{lem}{Lemma}[section]
\newtheorem{prop}[lem]{Proposition}
\newtheorem{cor}[lem]{Corollary}
\newtheorem{thm}[lem]{Theorem}

\theoremstyle{definition}

\newtheorem*{remark}{Remark}
\newtheorem{numremark}[lem]{Remark}

\newtheorem{example}[lem]{Example}

\usepackage{pstricks}

\begin{document}

\title{The cohomology ring of some Hopf algebras}

\author[Erdmann]{Karin Erdmann}
\address{Karin Erdmann\\
Mathematical Institute\\ 
24--29 St.\ Giles\\
Oxford OX1 3LB\\ 
England}
\email{erdmann@maths.ox.ac.uk}
\author[Solberg]{\O yvind Solberg}
\address{\O yvind Solberg\\
Institutt for matematiske fag\\
NTNU\\ 
N--7491 Trondheim\\ 
Norway}
\email{oyvinso@math.ntnu.no}
\author[Wang]{Xingting Wang}
\address{Xingting Wang\\
Department of Mathematics\\
Temple University, Philadelphia}
\email{wangxingting84@gmail.com}

\date{\today}

\begin{abstract}
Let $p$ be a prime, and $k$ be a field of characteristic 
$p$. We investigate the algebra structure and the structure of
the cohomology ring for the connected Hopf algebras of dimension $p^3$, 
which appear in the classification obtained  in \cite{NWW}. The list consists
of 23 algebras together with two infinite families. We identify the Morita
type of the algebra, and in almost all cases 
this is sufficient to clarify the structure  of the cohomology ring. 
\end{abstract}
\maketitle

\section{Introduction}

The classification of finite-dimensional Hopf algebras in
characteristic zero is well investigated by many researchers, see
survey paper \cite{An}. While this work is stimulating in its own
right, many Hopf algebras of interest are, however, defined over a
field $k$ of positive characteristic, where the classification is much
less known. Some work has been done in this direction. In \cite{Sc},
Scherotzke classified finite-dimensional pointed rank one Hopf
algebras in positive characteristic which are generated by group-like
and skew-primitive elements. Many Hopf algebras in positive
characteristic also come from Nichols algebras which are of much
interest, see \cite{CLW,HW}.

We recall that a Hopf algebra is called connected if it has only one
non-isomorphic simple comodule, i.e., its coradical is $k$. Note that
finite-dimensional connected Hopf algebras only appear in
characteristic $p$, for instance, group algebras of finite $p$-groups,
restricted universal enveloping algebras of restricted Lie algebras,
finite connected group schemes and others. 
In \cite{Hen}, all graded cocommutative
connected Hopf algebras of dimension less than or equal to $p^3$ are
classified. In the recent work of Nguyen, Wang and the third author
\cite{NWW}, connected Hopf algebras of dimension $p^3$ are classified over an
algebraically closed field $k$ of characteristic $p$ under some
assumption on the primitive space of these Hopf algebras. The result
is given below in Theorem \ref{thm:classification}.

We are interested in understanding the structure of the cohomology
rings for these algebras. This only depends on the algebra, not on the
Hopf structure.  This motivates our work, namely we analyse the
algebras, for those which are not local we obtain a presentation by
quiver and relations. This allows us to obtain the cohomology rings.

Section 2 contains the description of the Hopf algebras in question,
and a discussion of antipode, and Namakyama automorphism. Furthermore,
we state our results on the algebra structure. It turns out that the
algebras fall into six different classes, and in sections 3 to 8 we
deal with these.  Section 9 contains a reduction result which relates
finite generation of ext algebras related via adjoint functors, this
is more general.  Section 10 shows that some of the algebras in our
list can be viewed as twisted tensor products (in the sense of
\cite{CR}), and using Section 9 it follows that their cohomology is
Noetherian (with one exception).  Section 11 gives the cohomology
rings of the various algebras occuring in the classification.

For general background on quivers, path algebras and admissible
quotients of path algebras we refer the reader to \cite{ASS,ARS}. For
basic homological algebra we refer the reader to \cite{Ben}.

\section{Connected Hopf algebras of dimension $p^3$}

In this section we give the classification of connected Hopf
algebras of dimension $p^3$ with an explanation of the origin of the
classification from the primitive space of the Hopf algebras.  We
recall some facts about the antipode and the Nakayama automorphism of
finite dimensional Hopf algebras in general, and we apply this
knowledge to the algebras in the classification to describe the
antipode and the order of the Nakayama automorphism.  Finally we
explain how we classify them as algebras, which is the base of our
further investigations.  The classification theorem below assumes $k$ is
algebraically closed. In this paper, the field is usually arbitrary of
characteristic $p$. 

\subsection{Classification of connected Hopf algebras of
  dimension $p^3$}

The Hopf algebras in the classification of finite dimensional
connected Hopf algebras of dimension $p^3$ are always presented in the
form of $k\langle x, y, z\rangle/I$, where $I$ is an ideal generated
by relations. The comultiplication is given by
\begin{align}
\Delta(x) & = x \otimes 1 + 1 \otimes x,\notag\\
\Delta(y) & = y \otimes 1 + 1 \otimes y + Y,\notag\\
\Delta(z) & = z \otimes 1 + 1 \otimes z + Z,\notag
\end{align}
for some elements $Y$ and $Z$ in
$(k\langle x, y, z\rangle/I) \otimes (k\langle x, y, z\rangle/I)$. In
Theorem \ref{thm:classification}, 
proved in \cite{NWW}, the generators for the relations and
the ideal $I$ and the elements $Y$ and $Z$ are given explicitly. If
$Y$ and/or $Z$ are not given, then they are zero.  In order to express
$Y$ and $Z$, we use the notation
\[\omega(t) = \sum_{i=1}^{p-1} \frac{(p -1)!}{i!(p-i)!} t^i\otimes
  t^{p - i}.\]

\begin{thm}\label{thm:classification}
Let $k$ be an algebraically closed field of characteristic $p$.
The connected Hopf algebras over $k$ of dimension $p^3$ are precisely 
\begin{description}
\item[A1] $k[x,y,z]/\langle x^p - x, y^p - y, z^p - z\rangle$
with $Y = \omega(x)$ and $Z = \omega(x)[y\otimes 1 + 1\otimes y +
\omega(x)]^{p -1} + \omega(y)$. 
\item[A2] $k[x,y,z]/\langle x^p, y^p - x, z^p - y\rangle$. 
\item[A3] $k[x,y,z]/\langle x^p, y^p , z^p\rangle$. 
\item[A4] $k[x,y,z]/\langle x^p, y^p , z^p - x\rangle$. 
\item[A5, $p = 2$] $k\langle x,y,z\rangle / \langle x^2, y^2, [x,y],
[x,z], [y, z] - x, z^2 + xy\rangle$. 
\item[A5, $p > 2$] $A(\beta) = k\langle x, y, z\rangle/\langle x^p, y^p, [x,y],
  [x,z], [y, z] - x, z^p +x^{p - 1}y -\beta x\rangle$
for some $\beta\in k$ with $Y = \omega(x)$ and $Z = \omega(x)(y\otimes 1 +
1\otimes y)^{p-1} + \omega(y)$. 
Any two  $A(\beta)$ and $A(\beta')$ are isomorphic as Hopf algebras if and
only if $\beta' = \gamma\beta$ for some $(p^2 + p - 1)$-th root of
unity $\gamma$.  
\item[B1] $k\langle x, y, z\rangle/\langle [x,y] - y, [x,z], [y,z], x^p
- x, y^p, z^p\rangle$ with $Z = \omega(y)$.  
\item[B2] $k\langle x, y, z\rangle/\langle [x,y] - y, [x,z], [y,z] -
yf(x), x^p-x,y^p,z^p-z\rangle$, where
$f(x)=\sum_{i=1}^{p-1}(-1)^{i-1}(p-i)^{-1}x^i$ and $Z = \omega(x)$. 
\item[B3] $k\langle x, y, z\rangle/\langle [x,y]-y,[x,z]-z, [y,z]-y^2,
x^p-x, y^p, z^p\rangle$ with $Z = -2x\otimes y$ for $p > 2$.  
\item[C1] $k[x,y,z]/(x^p - x, y^p - y, z^p - z)$. 
\item[C2] $k[x,y,z]/(x^p - y, y^p - z, z^p)$.  
\item[C3] $k[x,y,z]/\langle x^p, y^p - z, z^p\rangle$. 
\item[C4] $k[x,y,z]/\langle x^p, y^p, z^p\rangle$.  
\item[C5] $k\langle x, y,z\rangle/\langle [x,y] - z, [x,z], [y,z], x^p,
y^p,z^p\rangle$.  
\item[C6] $k\langle x, y,z\rangle/\langle [x,y] - z, [x,z], [y,z], x^p
- z, y^p,z^p\rangle$ for $p>2$.  
\item[C7] $k[x,y,z]/\langle x^p, y^p, z^p-z\rangle$.  
\item[C8] $k[x,y,z]/\langle x^p - y, y^p, z^p-z\rangle$. 
\item[C9] $k[x,y,z]/\langle x^p, y^p - y, z^p-z\rangle$. 
\item[C10] $k\langle x, y,z\rangle/\langle [x,y] - z, [x,z], [y,z], x^p,
y^p,z^p - z\rangle$. 
\item[C11]  $k\langle x, y,z\rangle/\langle [x,y] - y, [x,z], [y,z], x^p
- x, y^p,z^p\rangle$. 
\item[C12] $ k\langle x, y,z\rangle/\langle [x,y] - y, [x,z], [y,z], x^p
- x, y^p - z,z^p\rangle$.
\item[C13] $k\langle x, y,z\rangle/\langle [x,y] - y, [x,z], [y,z], x^p
- x, y^p,z^p - z\rangle$. 
\item[C14] $ k\langle x, y,z\rangle/\langle [x,y] - y, [x,z], [y,z], x^p
- x, y^p - z, z^p - z\rangle$. 
\item[C15] $ k\langle x, y,z\rangle/\langle [x,y] - z, [x,z] - x, [y,z] +
y, x^p, y^p,z^p - z\rangle$ for $p>2$. 
\item[C16]
  $C(\lambda,\delta) = k\langle x, y,z\rangle/\langle [x,y], [x,z] -
  \lambda x, [y,z] - \lambda^{-1}y, x^p, y^p, z^p - \delta z\rangle$
  for some $\lambda\in k^\times$ such that
  $\delta = \lambda^{p-1} = \pm 1$.  Any two
  $C(\lambda_1,\delta_1)$ and $C(\lambda_2,\delta_2)$ are isomorphic
  as Hopf algebras if and only if $\delta_1=\delta_2$ and $\lambda_1=\lambda_2$ or
  $\lambda_1\lambda_2=1$.
\end{description}
\end{thm}
\begin{remark}
  1) The antipode of the above Hopf algebras always exists, because
  the coalgebra structure is connected.

2) Let $H$ be a connected Hopf algebra of dimension $p^3$ and
    denote by 
    \[P(H) =\{h\in H\mid\Delta(h)=h\otimes 1+h\otimes 1\},\] 
    the \emph{primitive space of $H$}.  The isomorphism classes of $H$
    given above satisfy the following conditions:
    \begin{enumerate}[(\textbf{\Alph*})] 
      \item $\dim_k P(H) = 1$.
      \item $\dim_k P(H) = 2$ with non-commuting elements.
      \item $\dim_k P(H) = 3$.
   \end{enumerate}

   3) Hopf algebras of type $\textbf{C}$ are exactly the restricted
   universal enveloping algebras of restricted Lie algebras of
   dimension $3$. It includes all the $p$-nilpotent restricted Lie
   algebras of dimension $3$ classified in \cite[Theorem 2.1
   (3/1),(3/2)]{SU}, i.e., (up to isomorphisms) (3/1) (a) is
   \textbf{C4}, (3/1) (b) is \textbf{C3}, (3/1) (c) is \textbf{C2};
   (3/2) (a) is \textbf{C5} ($p\ge 3$) and \textbf{C10} ($p=2$), and (3/2) (b) is
   \textbf{C6}.
\end{remark}

\subsection{Antipode and Nakayama automorphism}
In this subsection we review some known results on the antipode and
the Nakayama automorphism of finite dimensional Hopf algebras in
general, and then apply them to the algebras in Theorem
\ref{thm:classification}.
 
Let $H=(H,m,u,\Delta,\epsilon,S)$ be any finite-dimensional Hopf
algebra over an arbitrary base field $k$.  A \emph{left integral in
  $H$} is an element of $H$, usually denoted by $\L$, such that
$h\Lambda=\epsilon(h)\Lambda$ for all $h\in H$; and a \emph{right
  integral in $H$} is an element $\Lambda'\in H$ such that
$\Lambda'h=\epsilon(h)\Lambda'$ for all $h\in H$. The space of left
integrals and the space of right integrals are denoted by $\int^l_H$
and $\int^r_H$, respectively. A well-known result of Larson and
Sweedler shows that $\dim \int^l_H=\dim \int^r_H=1$. We say $H$ is
\emph{unimodular} if $\int^l_H=\int^r_H$.

For any algebra map $\alpha\colon H\to k$, we can define the left
winding automorphism of $H$ associated to $\alpha$ as
$\Xi^l[\alpha](h)=\sum \alpha(h_1)h_2$ for any $h\in H$. Similarly,
the right winding automorphism of $H$ associated to $\alpha$ can be
defined as $\Xi^r[\alpha](h)=\sum h_1\alpha(h_2)$ for any $h\in
H$. Note that $\int^l_H\subset H$ is a right $H$-submodule of
$H$. Since it is one-dimensional, the right $H$-module structure gives
an algebra map $\alpha\colon H\to k$ such that
$\Lambda h=\alpha(h)\Lambda$ for $0\neq \Lambda\in \int^l_H$ and
$h\in H$. We will use the term left winding automorphism of $\int^l_H$
instead of $\alpha$.

The following results are due to Larson-Sweedler \cite{LS},
Pareigis \cite{Pare}, and Brown-Zhang \cite{BZ}.

\begin{thm}\label{thm:nakauto}
Let $H$ be any finite-dimensional Hopf algebra. Then 
\begin{enumerate}[\rm(a)]
\item $H$ is Frobenius with a nondegenerate associated bilinear form
  $\langle-,-\rangle\colon H\otimes H\to k$ given by
  $\langle a,b\rangle=\lambda(ab)$, where
  $0\neq \lambda\in \int^l_{H^*}$ and $a,b\in H$.
\item The Nakayama automorphism is given by $S^2\xi$, where $\xi$ is
  the left winding automorphism of $\int^l_H$.
\item $H$ is symmetric if and only if $H$ is unimodular and $S^2$ is
  inner.
\end{enumerate}
\end{thm}
This has the following consequence for finite dimensional involutory
Hopf algebras.  Recall that a Hopf algebra $H$ with antipode $S$ is
\emph{involutory} if $S^2 = 1$.

\begin{cor}\label{cor:nakauto}
The following finite-dimensional involutory Hopf algebras are symmetric:
\begin{enumerate}[\rm(a)]
\item commutative algebras,
\item group algebras,
\item local algebras,
\item semisimple algebras.
\end{enumerate}
\end{cor}
\begin{proof} 
  Parts (a) and (b) are clear.  Parts (c) and (d) hold because any local or semisimple
  Hopf algebra is unimodular.
\end{proof}

The next result for restricted Lie algebras is from \cite{Sch, Hum}.
\begin{thm}\label{thm:restricteL}
  Let $\mathfrak g$ be a finite dimensional restricted Lie algebra
  over $k$ of characteristic $p>0$. Then the restricted enveloping
  algebra $u(\mathfrak g)$ is Frobenius with Nakayama automrophism
  $\sigma$ given by $\sigma(x)=x+\Tr(\ad x)$ for all
  $x\in \mathfrak g$. As a consequence, $\sigma^p=1$ and
  $u(\mathfrak g)$ is symmetric if and only if
  $\Tr(\ad x)=0$ for all $x\in \mathfrak g$.
\end{thm}

A direct calculation by using the antipode axiom 
\[m(S\otimes 1)\Delta=m(1\otimes S)\Delta=u\epsilon\] 
yields that the antipode of Hopf algebras for all \textbf{A},
\textbf{B}, \textbf{C} types is given by 
\[S(x)=-x,\qquad S(y)=-y,\qquad S(z)=-z\]
except for \textbf{B3} whose antipode is given by
\[S(x)=-x,\qquad S(y)=-y,\qquad S(z)=-z-2xy.\]
Using this we have the following.
\begin{prop}
  For Hopf algebras listed in \textup{Theorem
    \ref{thm:classification}}, the following are true.
\begin{enumerate}[\rm(a)]
\item All Hopf algebras are involutory except for \textup{\textbf{B3}}
  where $S^{2p}=1$.
\item Hopf algebras of type \textup{\textbf{A}} are unimodular and
  symmetric.
\item The Hopf algebras of type \textup{\textbf{B1}} and
  \textup{\textbf{B3}} are not symmetric.
\item The Hopf algebras \textup{\textbf{C1}}--\textup{\textbf{C10}} and
  \textup{\textbf{C15}} are symmetric,
  \textup{\textbf{C11}}--\textup{\textbf{C14}} are not symmetric, and
  \textup{\textbf{C16}} is symmetric if and only if $\lambda^2=-1$.
\end{enumerate}
\end{prop}
\begin{proof}
  (a) is clear since for \textbf{B3}, we have $S^{2n}(z)=z+2ny$ for
  all $n\ge 1$. 

  (b) The algebras \textbf{A1}--\textbf{A4} are commutative. One checks easily that
   \textbf{A5} is local. Then it follows from Corollary \ref{cor:nakauto}(c).

  (c) \textbf{B1} and \textbf{C11} are isomorphic as augmented algebras so
  it follows from (d) later. For \textbf{B3}, one sees that
  $S^2(z)=z+xy$ is not inner.

  (d) This follows from Theorem \ref{thm:restricteL}.
\end{proof}

\begin{remark}
  From a Hopf algebra view point it is not known whether \textbf{B2}
  is symmetric or not because it is hard to compute its left or right
  integral and to check whether it is unimodular.  However we shall
  show that through finding another presentation of algebras of type
  \textbf{B2} none of them are symmetric (See Proposition
  \ref{prop:selfinjNaka} (a)). 
\end{remark}

\subsection{Algebra structure}
The cohomology ring of a finite dimensional Hopf algebra only depends
on the algebra structure, not the Hopf structure.  So in order to
determine the structure of the cohomology ring it is enough to study
the algebra structure.  We will do this in the following six sections,
where we show that the algebras in this classification falls into $6$
different classes, (0) semisimple algebras, (1) group algebras
(tensored or direct sum with a semisimple algebra), (2) (direct sums
of) selfinjective Nakayama algebras, (3) enveloping algebra of
restricted Lie algebras, (4) coverings of local algebras, (5) other
local algebras.  We will show that
\begin{center}
\begin{tabular}{|l|l|}\hline
  Semisimple algebras & \textbf{A1} = \textbf{C1} \\ \hline
  Group algebras &  \parbox{6cm}{\textbf{A2} = \textbf{C2}, \textbf{A3} =
                   \textbf{C4}, \textbf{A4} = \textbf{C3}, 
                   \textbf{C7}, \textbf{C8}, \textbf{C9}, \textbf{C10}}
  \\ \hline
Selfinjective Nakayama algebras & \textbf{B2}, \textbf{C12}, \textbf{C13},
                      \textbf{C14}\\ \hline
  \parbox{4cm}{Enveloping algebra of restricted Lie algebras} & \textbf{C5},
                                                  \textbf{C6},
                                                  \textbf{C15}\\ \hline
Coverings of local algebras & \textbf{B1} = \textbf{C11}, \textbf{B3}, \textbf{C16}\\\hline
  Other local algebras  &  \textbf{A5}\\ \hline    
\end{tabular}
\end{center}
Note that here we focus on the algebra structure (and do not consider  the
comultiplication).

\section{Semisimple algebras}

In this section we classify the semisimple algebras occurring among
the algebras in the classification of the finite dimensional connected
Hopf algebras of dimension $p^3$.

We show the following. 
\begin{prop} Let $k$ be a field of characteristic $p$.  The algebras of
  type \textup{\textbf{A1}} and \textup{\textbf{C1}} are equal, and they are isomorphic to
  $k^{p^3}$.
\end{prop}
\begin{proof}
  The algebra of type \textbf{A1} (or \textbf{C1}) are given as
  \[\L = k[x,y,z]/\langle x^p - x, y^p - y, z^p - z\rangle\] 
  for a field $k$ of characteristic $p$.  The polynomial $u^p -u$ in $u$
  has $p$ different roots in $\mathbb{Z}_p$, so that
  $k[u]/\langle u^p-u\rangle \simeq k^p$.  Furthermore we have that
\[ \L\simeq  k[x]/(x^p-x) \otimes_k k[y]/(y^p-y)\otimes_k k[z]/(z^p-z)
  \simeq k^{p^3}.\]
\end{proof}

\begin{numremark}\label{rem:idempotents}
As we already said, the $p$-dimensional algebra 
$B = k[u]/\langle u^p-u\rangle$ is isomorphic to $k^p$ for $k$ of characteristic $p$.
There is an elementary explicit formula for the orthogonal primitive
idempotents of this algebra, and we will use this later. Namely 
for $r\in \mathbb{Z}_p$ let
$$e_r:= \prod_{\stackrel{s=0}{s\neq r}}^{p-1} \frac{(u-s)}{(r-s)}.
$$
Then $ue_r = re_r$, and the $e_r$ for $r\in \mathbb{Z}_p$ are
pairwise orthogonal idempotents and their sum is the identity of the algebra.
We mention two consequences.

(1) Suppose $0\neq \zeta \in B$ and $u\zeta = r\zeta$, then $\zeta$ is
a scalar multiple of $e_r$. Namely, we have $\zeta = \sum_i \lambda_ie_i$ for
$\lambda_i\in k$, and then 
$0 = (u-r)\zeta = \sum_i \lambda_i(u-r)e_i$. From the formula, $(u-r)e_i=0$ 
for $i=r$ and otherwise it is a non-zero scalar multiple of $e_i$, and
the claim follows.

(2) Suppose $B$ is contained in some algebra $\Lambda$, and $y\in \Lambda$
satisfies $[u,y] = y$. Then $ye_i = e_{i+1}ye_i$. Namely we have
$uy = y(u+1)$ and therefore $u(ye_i) = y(u+1)e_i = (i+1)ye_i$. 
Variations of this argument will be used later.
\end{numremark}

\section{Group algebras}

In this section we find group algebras occurring in the classification
of the finite dimensional connected Hopf algebras of dimension $p^3$.
Here we denote by $C_n$ the cyclic group of order $n$, and for a
finite group $G$ we denote by $kG$ the group algebra of $G$ over a
field $k$.

\begin{prop} Let $k$ be a field of characteristic $p$. 
\begin{enumerate}[\rm(a)]
\item  The algebras  \textup{\textbf{A2}} and \textup{\textbf{C2}} are equal, and
  they are isomorphic to $kC_{p^3}$.
\item \sloppypar The algebras  \textup{\textbf{A3}} and \textup{\textbf{C4}} are equal, and
  they are isomorphic to 
\[k(C_p\times C_p\times C_p).\]
\item The algebras  \textup{\textbf{A4}} and \textup{\textbf{C3}} are
  isomorphic, and they are isomorphic to
\[k(C_p\times C_{p^2}).\] 
\item The algebra \textup{\textbf{C7}} is isomorphic to
the direct sum of  $p$ copies of $k(C_p\times C_p)$.
\item The algebra  \textup{\textbf{C8}} is isomorphic to the direct sum of 
 $p$ copies of $kC_{p^2}$. 
\item The algebra \textup{\textbf{C9}} is isomorphic to the direct sum of $p^2$ copies of
  $kC_p$.
\item The algebra  \textup{\textbf{C10}} is isomorphic to 
the direct sum of $k(C_p\times
  C_p)$ and  $p - 1$ copies of $M_p(k)$. 
\end{enumerate}
\end{prop}
\begin{proof}
  (a) The algebras of type \textbf{A2} and \textbf{C2} are given as (with a possible change of variables)
  \[\L = k[x,y,z]/\langle x^p - y, y^p - z, z^p\rangle.\]  The algebra homomorphism
  $k[u] \to \L$ sending
  $u\mapsto x + \langle x^p -y, y^p - z, z^p\rangle$ is surjective
  with kernel $\langle u^{p^3}\rangle$.  Hence
  $\L \simeq k[u]/\langle u^{p^3}\rangle$, which is isomorphic to
  $kC_{p^3}$.

  (b) The algebras of type \textbf{A3} and \textbf{C4} are given as
  $\L = k[x,y,z]/\langle x^p, y^p, z^p\rangle$.  Then
\[\L \simeq k[x]/\langle x^p\rangle \otimes_k k[y]/\langle y^p\rangle
\otimes_k k[z]/\langle z^p\rangle.\]
This is in turn isomorphic to $k(C_p\times C_p\times C_p)$. 

(c) The algebra of type \textbf{A4} and \textbf{C3} are given as (with a possible change of variables)
$\L = k[x,y,z]/\langle x^p, y^p - z, z^p\rangle$.  Then
\begin{align}
\L & \simeq k[x]/\langle x^p\rangle \otimes_k k[y, z]/\langle y^p - z,
z^p\rangle\notag\\
& \simeq k[x]/\langle x^p\rangle \otimes_k k[u]/\langle
u^{p^2}\rangle,\notag
\end{align}
since the algebra homomorphism $k[u] \to k[y, z]/\langle y^p -z,
z^p\rangle$ sending $u\mapsto y + \langle y^p -z, z^p\rangle$ is
surjective with kernel $\langle u^{p^2}\rangle$.  This algebra is
isomorphic to $k(C_p\times C_{p^2})$. 

(d) The algebra of type \textbf{C7} is given as
$\L = k[x,y,z]/\langle x^p, y^p, z^p-z\rangle$.  Then
\begin{align}
\L &\simeq k[x]/\langle x^p\rangle \otimes_k k[y]/\langle y^p\rangle
\otimes_k k[z]/\langle z^p - z\rangle\notag\\
& \simeq k[x]/\langle x^p\rangle \otimes_k k[y]/\langle y^p\rangle
\otimes_k k^p\notag
\end{align}
The last algebra is isomorphic to $k(C_p\times C_p)\otimes_k k^p$,
which is isomorphic to $p$ copies of $k(C_p\times C_p)$. 

(e) The algebra of type \textbf{C8} is given as
$\L = k[x,y,z]/\langle x^p - y, y^p, z^p-z\rangle$.  Then
\begin{align}
\L & \simeq k[x,y]/\langle x^p - y, y^p\rangle \otimes_k k[z]/\langle
z^p - z\rangle\notag\\
 & \simeq k[x,y]/\langle x^p - y, y^p\rangle \otimes_k k^p\notag\\
 & \simeq k[u]/\langle u^{p^2}\rangle\otimes_k k^p,\notag
\end{align}
since the algebra homomorphism
$k[u] \to k[x,y]/\langle x^p -y, y^p\rangle$ sending
$u\mapsto x + \langle x^p -y, y^p\rangle$ is surjective with kernel
$\langle u^{p^2}\rangle$.  This last algebra occurring above is in
turn isomorphic to $kC_{p^2}\otimes_k k^p$, which is isomorphic to $p$
copies of $kC_{p^2}$. 

(f) The algebra of type \textbf{C9} is given as
$\L = k[x,y,z]/\langle x^p, y^p - y, z^p-z\rangle$.  Then
\begin{align}
\L & \simeq k[x]/\langle x^p\rangle \otimes_k k[y]/\langle y^p - y\rangle
\otimes_k k[z]/\langle z^p - z\rangle\notag\\
& \simeq k[x]/\langle x^p\rangle \otimes_k k^{p^2}\notag
\end{align}
This last algebra is isomorphic to $kC_p\otimes_k k^{p^2}$, which in
turn is isomorphic to $p^2$ copies of $kC_p$. 

(g) The algebra of type \textbf{C10} is given as
\[\L = k\langle x, y,z\rangle/\langle [x,y] - z, [x,z], [y,z], x^p,
  y^p,z^p - z\rangle.\] 
It is clear from the relations that the subalgebra generated by
$z$ is in the center of $\L$ and is giving rise to $p$
orthogonal central idempotents $\{e_i\}_{i=0}^{p-1}$ with
$ze_i = ie_i$ as in Remark \ref{rem:idempotents}.  Then $\L =
\oplus_{i=0}^{p-1} \L e_i$, and we  
have that $ie_i = ze_i = [x,y]e_i$.  Furthermore
$\L e_i$ is generated by $\{xe_i, ye_i\}$, hence we
have a surjective homomorphism of rings
$\varphi_i\colon k\langle x,y\rangle \to \L e_i$ sending
$x\mapsto xe_i$ and $y\mapsto ye_i$.  The kernel of
$\varphi_i$ is generated by the elements
$x^p, y^p, [x,y] - i$ giving that
$\L e_i\simeq k\langle x,y\rangle/\langle x^p,y^p, [x,y] - i\rangle$.

For $i=0$ we get that $\L e_0 \simeq k[x, y]/\langle x^p,
y^p\rangle\simeq k[x]/\langle x^p\rangle\otimes_k k[y]/\langle
y^p\rangle$, which is isomorphic to $k(C_p\times C_p)$.  

Now consider the case when $i$ is in $[1,\ldots,p-1]$.  We have that 
\[\L \simeq 
  k\otimes_{\mathbb{Z}_p} \mathbb{Z}_p\langle x, y,z\rangle/\langle
  [x,y] - z, [x,z], [y,z], x^p, y^p,z^p - z\rangle,\]
and all the arguments we used above are valid over $\mathbb{Z}_p$.  So
we first analyze the algebra over $\mathbb{Z}_p$ and then induce up to
$k$.  So to reduce the amount of notation we also call the algebra
$\L$ when considering it given over the field $\mathbb{Z}_p$.  

Any element in $\L e_i$ can be written as a linear combination of
$\{ x^ry^se_i\}_{r,s=0}^{p-1,p-1}$.  Given any
non-zero ideal $I$ in $\L e_i$, we can assume that
$x^{p-1}y^{p-1}e_i$ is in $I$ by multiplying
with a power of $x$ from the left and a power of
$y$ from the right.  Using the identity
$xye_i - ie_i = yxe_i$, we
obtain that
$y^se_ix = xy^se_i - s
iy^{s-1}e_i$ for $s\geqslant 1$.  Starting with multiplying
the element $x^{p-1}y^{p-1}e_i$ with
$x$ from the right, we infer that
$x^{p-1}y^{p-2}e_i$ is in $I$.  Inductively it
follows that $x^{p-1}e_i$ (and by similar arguments that
$y^{p-1}e_i$) is in $I$.  Consequently,
$x^{p-1}y^se_i$ is in $I$ for all
$1\leqslant s\leqslant p-1$.  By multiplying $x^{p-1}e_i$
from the left with $y$, we infer that
$x^{p-2}e_i$ is in $I$.  Continuing this way we obtain
that $e_i$ is in $I$ and $I = \L e_i$.  It follows that $\L e_i$ is a
simple ring for $i$ in $[1,\ldots,p-1]$.  By Wedderburn-Artin $\L e_i$
is isomorphic to $M_n(D)$ for $n\geq 1$ and a division ring $D$.
Using that $\dim_{\mathbb{Z}_p} \L e_i = p^2$, one can show that the
only possibility for $\L e_i$ is $M_p(\mathbb{Z}_p)$.  Inducing up to
$k$, we infer that the algebra of type \textbf{C10} is isomorphic to
$k(C_p\times C_p) \oplus M_p(k)^{p-1}$.  In fact, a simple $\L e_i$-module $S$
with basis $\mathcal{B} = \{b_0,b_1,\ldots,b_{p-1}\}$ can be
constructed in the following way.  Let $xb_j = 
\begin{cases} 
b_{j - 1}, & j \neq 0\\
0, & j = 0
\end{cases}$
and 
$yb_j = 
\begin{cases} 
\lambda_{j+1}b_{j + 1}, & j \neq p - 1\\
0, & j = p - 1
\end{cases}$ for $\lambda_{j+1}=-(j+1)i$.  Then one can show
that the vector space spanned by $\mathcal{B}$ is a simple
$\L e_i$-module.
\end{proof}

\section{Selfinjective Nakayama algebras}

In this section we find selfinjective Nakayama algebras occurring in the
classification of the finite dimensional connected Hopf algebras of
dimension $p^3$.  Here we denote by $\widetilde{\mathbb{A}}_n$ the
quiver consisting of $n + 1$ vertices with one arrow starting in each
vertex forming an oriented cycle.  For a field $k$ and an integer
$n\geq 1$ we denote by $J$ the ideal in $k\widetilde{\mathbb{A}}_n$
generated by the arrows. 

\begin{prop}\label{prop:selfinjNaka}
Let $k$ be a field of characteristic $p$. 
\begin{enumerate}[\rm(a)]
\item The algebra of type \textup{\textbf{B2}} is isomorphic to
  $k\widetilde{\mathbb{A}}_{p^2-1}/J^p$, which is not a symmetric algebra. 
\item The algebra of type \textup{\textbf{C12}} is isomorphic to
  $k\widetilde{\mathbb{A}}_{p-1}/J^{p^2}$. 
\item The algebra of type \textup{\textbf{C13}} is isomorphic to a direct sum
  of $p$ copies of the selfinjective Nakayama algebra
  $k\widetilde{\mathbb{A}}_{p-1}/J^p$ of dimension $p^2$. 
\item The algebra of type \textup{\textbf{C14}} is isomorphic 
$k\widetilde{\mathbb{A}}_{p-1}/J^p$  direct sum $p-1$ copies of
$M_p(k)$. 
\end{enumerate}
\end{prop}
\begin{proof}
  (a) The algebra of type \textbf{B2} is given as
  \[\L=k\langle x, y, z\rangle/\langle [x,y] - y, [x,z], [y,z] - yf(x),
    x^p-x,y^p,z^p-z\rangle,\] 
  where $f(x)=\sum_{i=1}^{p-1}(-1)^{i-1}(p-i)^{-1}x^i$.  We have that
  $yx = xy - y$, $zx = xz$
  and $zy = yz - yf(x)$.  This implies that any
  monomial in the variables
  $\{x, y, z\}$ can be written as a
  linear combination of monomials of the form $x^ry^sz^t$
  for $r$, $s$ and $t$ in $\{0,1,\ldots,p-1\}$.  Hence the set
  $\mathcal{B} = \{x^ry^sz^t\}_{r,s,t=1}^{p-1}$ spans $\L$,
  which consists of $p^3$ elements.  Since the dimension of $\L$ is
  $p^3$, the set $\mathcal{B}$ is a basis for $\L$.

  All the relations for the algebra are homogeneous in $y$.  It
  follows from this that
  $\langle y\rangle^i = \langle y^i\rangle$.
  Since $y^p = 0$, it follows that
  $\langle y\rangle^p=(0)$.  Furthermore,
  \[\L/\langle y\rangle \simeq k[x]/\langle
  x^p-x\rangle\otimes_k k[z]/\langle z^p-z\rangle\simeq k^{p^2}.\] 
It follows that $\rad\L = \langle y\rangle$.  

The subalgebra of $\L$ generated by $x$ and $z$ is
  \[k\langle x, z\rangle/\langle x^p - x, [x,z], z^p -z\rangle \simeq
    k[x]/\langle x^p-x\rangle\otimes_k k[z]/\langle z^p - z\rangle,\]
  which is commutative.  Let $e_0, \ldots, e_{p-1}$ be the orthogonal
  primitive idempotents of the algebra generated by $x$ so that
  $xe_i=ie_i$.  Let also $g_0, \ldots, g_{p-1}$ be the
  orthogonal primitive idempotents of the algebra generated by $z$ so
  that $zg_j=jg_j$. Then these span a commutative
  semisimple subalgebra of dimension $p^2$ of $\L$, with orthogonal
  primitive idempotents $e_ig_j$.  These correspond to the vertices in
  the quiver we shall construct (as they are liftings of the primitive
  idempotents in $\L/\rad \L$).  
  
  These orthogonal idempotents decompose $\L = \oplus_{i, j} \L e_ig_j$
  into a direct sum of $p^2$ left modules.  The summand $\L e_ig_j$ has
  basis $\{ y^r e_ig_j\}_{r = 0}^{p-1}\}$ and the radical is generated
  by $ye_ig_j$ (and $\L e_ig_j$ is indecomposable as a left module).
  Hence $\rad\L/\rad^2\L$ is generated by
  $\{ye_ig_j\}_{i,j=0}^{p-1,p-1}$, and they correspond to the arrows
  in the quiver we shall construct.   

  We claim that $ye_ig_j = ( e_{i+1}g_{j - f(i)})ye_ig_j(e_ig_j)$.
  Clearly $ye_ig_j = ye_ig_j(e_ig_j)$.  

  We have the relation $xy=y(x + 1)$ which implies that
  $x\cdot ye_i = (i + 1)ye_i$ and hence
  $x\cdot (ye_ig_j) = (i + 1)ye_ig_j$.  Using Remark
    \ref{rem:idempotents} we obtain that
\[e_{i+1}ye_ig_j = ye_ig_j.\]

We have   $zy = yz - y\cdot f(x)$ and $z, x$ commute, so 
\[z(yg_j) = yzg_j - yf(x)g_j = j\cdot yg_j - yg_j\cdot f(x)\]
and therefore
\begin{align}
z(ye_ig_j) = z(yg_je_i) & = (jyg_j)e_i - (yg_jf(x))e_i\notag\\
                                 & = jyg_je_i - yg_jf(i)e_i\notag\\
                                & = (j - f(i))\cdot (ye_ig_j)\notag
\end{align}
and again by Remark \ref{rem:idempotents}, we have that 
$yg_je_i =  g_{j - f(i)}\cdot yg_je_i$.  It follows that $ye_ig_j
= ( e_{i+1}g_{j-f(i)})ye_ig_j(e_ig_j)$. 

It follows from the above that if $Q$ is a quiver with $p^2$ vertices
$v_{i,j} = ye_ig_j$ and $p^2$ arrows
$a_{i,j}\colon v_{i,j}\to v_{i+1,j-f(i)}$, then $\L$ is a quotient of
$kQ$.  To show that $Q$ is $\widetilde{\mathbb{A}}_{p^2-1}$, we need to
show that the orbit of $(0,0)$ under the map
$(i,j)\mapsto (i+1, j - f(i))$ is all of $\Z_p\times \Z_p$.  After
$kp$ steps one gets
\begin{equation}\label{eq:orbit}
(kp, -f(1) - f(2) - \cdots - f(kp-1))
\end{equation}
To compute this, note that for $1\leq m\leq p-1$ one has 
$\sum_{j=1}^{p-1} j^m \equiv -1 \mod p$ if
$m = p-1$ and is zero otherwise (To see this, one can for example use
\cite[Lemma 4.3]{Gr}).  Then one gets that \eqref{eq:orbit} is equal
to $(0, -k)$ and it follows that the orbit of $(0,0)$ must have size
$p^2$ as required.  Hence the quiver $Q$ is
$\widetilde{\mathbb{A}}_{p^2-1}$ with the orientation as an oriented
cycle.  Since the Loewy length of the indecomposable projectives
$\L e_ig_j$ are all equal, the relation ideal is $J^p$.

A Nakayama algebra $k\widetilde{\mathbb{A}}_n/J^t$ is symmetric if and
only if $n+1$ divides $t - 1$ by \cite[Corollary IV.6.16]{SY}. Since
$p^2\nmid p - 1$ for every prime $p$, all the algebras of type
\textbf{B2} are not symmetric.

\begin{remark} 
  Note that the field $k$ can be arbitrary of characteristic $p$.  
\end{remark}

(b) The algebra of type \textbf{C12} is given as
\[\L = k\langle x, y,z\rangle/\langle [x,y] - y, [x,z], [y,z], x^p - x,
y^p - z,z^p\rangle.\]  
Any element in $\L$ can be written as a linear combination of elements
in the set $\{x^ry^s\}_{r,s=0}^{p-1,p^2-1}$.  Let
$\mathfrak{a} = \langle y\rangle$ in $\L$.  Since the
relations are homogeneous in $y$, it follows from the above that
$\mathfrak{a}$ is a nilpotent ideal in $\L$.  Furthermore it is easy
to see that
$\L/\mathfrak{a} \simeq k[x]/\langle x^p - x\rangle \simeq k^p$.  We
infer that $\rad\L = \mathfrak{a}$.  By the above comments we have
that $\rad\L$ has basis $\{ x^ry^s\}_{r=0, s\geqslant 1}^{p-1, p^2-1}$
(dimension $p^3-p$) and $\rad^2\L$ has basis
$\{ x^ry^s\}_{r=0, s\geqslant 2}^{p-1, p^2-1}$ (dimension $p^3-2p$).
The $\rad\L/\rad^2\L$ has a basis given by the residue classes of the
elements $\{ x^ry\}_{r=0}^{p-1}$.  Hence $\L$ is isomorphic to a
quotient of a path algebra $kQ$ over $k$, where $Q$ has $p$ vertices
\[\left\{ v_\alpha = \frac{\prod_{\beta\in \mathbb{Z}_p\setminus
      \{\alpha\}} (x-\beta)}{\prod_{\beta\in \mathbb{Z}_p\setminus
      \{\alpha\}} (\alpha-\beta)}\right\}_{\alpha\in\mathbb{Z}_p}\]
and $p$ arrows given by $\{v_\alpha y\}_{\alpha\in\mathbb{Z}_p}$.
By part (2) of Remark \ref{rem:idempotents} we have that 
\[v_\alpha y = y v_{\alpha - 1},\]
so that $v_{\alpha + 1} y$ is an arrow from vertex $v_\alpha\to
v_{\alpha + 1}$ for all $\alpha$ in $\mathbb{Z}_p$. 
It follows that $\L\simeq k\widetilde{\mathbb{A}}_{p-1}/J^{p^2}$.  

(c) The algebra of type \textbf{C13} is given as
\[\L = k\langle x, y,z\rangle/\langle [x,y] - y, [x,z], [y,z], x^p - x,
y^p,z^p - z\rangle.\]  
This is isomorphic to the tensor product 
\[\L \cong k\langle x, y\rangle/\langle [x,y]-y, x^p-x, y^p\rangle
  \otimes_k k[z]/\langle z^p-z\rangle.\] The second tensor factor is
semisimple, isomorphic to $k^p$, hence $\L$ is the direct sum of $p$
copies of
$B = k\langle x, y\rangle/\langle [x,y]-y, x^p-x, y^p\rangle$.  Let
$\mathfrak{a} = \langle y\rangle$ in $B$.  Since the relations are
homogeneous in $y$, it follows from the above that $\mathfrak{a}$ is a
nilpotent ideal in $B$.  Furthermore it is easy to see that
$B/\mathfrak{a} \simeq k[x]/\langle x^p - x\rangle \simeq k^{p}$.  We
infer that $\rad B = \mathfrak{a}$.  The subalgebra generated by $x$
gives rise to a complete set of orthogonal idempotents $e_i$ with
$xe_i = ie_i$ (see Remark \ref{rem:idempotents}), so that
$B = \oplus_{i=0}^{p-1}Be_i$. Then $Be_i$ has basis
$\{ y^se_i\}_{s=0}^{p-1}$ and its radical is generated by $ye_i$. We
have $x(ye_i) = y(x+1)e_i = (i+1)ye_i$, hence by Remark
\ref{rem:idempotents} $ye_i = e_{i+1}ye_i$ and $ye_i$ gives an arrow
$e_i\to e_{i+1}$. It follows from this that $B$ is isomorphic to
  the algebra $k\widetilde{\mathbb{A}}_{p - 1}/J^p$ of dimension
$p^2$.

(d) The algebra of type \textbf{C14} is given as
\[\L = k\langle x, y,z\rangle/\langle [x,y] - y, [x,z], [y,z], x^p - x,
y^p - z, z^p - z\rangle.\]  
It is clear from the relations that the subalgebra generated by
$z$ is in the center of $\L$ and is giving rise to $p$
orthogonal central idempotents $\{e_i\}_{i=0}^{p-1}$ with
$ze_i = ie_i$.  Then $\L = \oplus_{i=0}^{p-1} \L e_i$.  We
claim that $\L e_0$ is isomorphic to
$k\widetilde{\mathbb{A}}_{p-1}/J^p$, while $\L e_i$ is isomorphic to
$M_p(k)$ for $i\ge1$.  To do this, we first consider the algebra $\L$ as
an algebra over the prime field $\mathbb{Z}_p$.  The idempotents
constructed above are given over this field.  By abuse of notation we
still denote the algebra by $\L$ when considered over the field $k' =
\mathbb{Z}_p$. 

The subalgebra generated by $x$ gives rise to complete set
of orthogonal idempotents $\{ f_i\}_{i=1}^{p-1}$ with
$xf_i = i f_i$.  This implies that
$\{y^re_if_s\}_{r,s=0}^{p-1,p-1}$ is basis for $\L e_i$.  Let $i \ge 1$,
then we want to show that $\L e_i$ is a simple ring.  Let $I$ be a
non-zero ideal in $\L e_i$ with a non-zero element
$m = \sum_{r,s=0}^{p-1,p-1} \alpha_{r,s}y^re_if_s$.  Assume that
$\alpha_{r_0,s_0} \neq 0$.  Then
$m f_{s_0} = \sum_{r=0}^{p-1} \alpha_{r,s_0}y^re_if_{s_0}$ is in $I$.
Since $f_jy = yf_{j-1}$, it follows that
$f_{r_0+s_0}m f_{s_0} = \alpha_{r_0,s_0}y^{r_0}e_if_{s_0}$ is in
$I$ and consequently $y^{r_0}e_if_{s_0}$ is in $I$.  Multiplying this
last element from the right with powers of $y$, we obtain that
$y^{r_s}e_if_s$ is in $I$ for some $r_s$ for $s=0,1,\ldots,p-1$.
Multiplying from the left with $y^{p-r_s}$ we obtain that $e_if_s$ is
in $I$ for $s=0,1,\ldots,p-1$, hence $e_i$ is in $I$ and $I = \L e_i$
and $\L e_i$ is a simple ring.  Furthermore, $\L e_i \simeq M_n(D)$
for some $n\geq 1$ and a division ring $D$.  A simple module over
$\L e_i$ then has dimension $n\dim_{k'} D$ over $k'$.  Since $\L e_i$ is
non-commutative and $D$ is commutative, $n > 1$.  Using similar
arguments as above one can show that $S_t =
k'\{y^re_if_t\}_{r=0}^{p-1}$ is a simple $\L e_i$-module.  Hence
$n\dim_{k'} D = p$.  It follows that $n=p$ and $D=k'$, and therefore
when inducing up to the field $k$ we get that $\L e_i\simeq M_p(k)$
for $i \ge1$.   

For $\L e_0$, let $\varphi\colon k\langle x,y\rangle \to \L e_0$ be
induced by the inclusion
$k\langle x, y\rangle \hookrightarrow k\langle x, y,z\rangle$.  Then
$\{[x,y]-y, x^p-x, y^p\}$ are contained in $\Ker \varphi$.  Since
\[\dim_k\L e_0 = \dim_k k\langle x,y\rangle/\langle [x,y]-y, x^p-x,
  y^p\rangle = p^2,\] 
it follows that
$\L e_0 \simeq k\langle x,y\rangle/\langle [x,y]-y, x^p-x,
y^p\rangle$, which we consider as an identification.  The ideal
generated by $y$ is the radical of $\L e_0$.  View the
primitive orthogonal idempotents $\{f_i\}_{i=0}^{p-1}$ as elements in
$\L e_0$.  Then $\{y^rf_s\}_{r,s=0}^{p-1,p-1}$ is a $k$-basis for
$\L e_0$. Therefore $\{f_s\}_{s=0}^{p-1}$ is a $k$-basis for
$\L e_0/\rad \L e_0$ and $\{ yf_s\}_{s=0}^{p-1}$ is a $k$-basis for
$\rad\L e_0/\rad^2 \L e_0$.  We let these sets define the vertices and
the arrows in a quiver $Q$, which is $\widetilde{\mathbb{A}}_{p-1}$.
All the indecomposable projective modules have the same Loewy length
and $\L e_0$ has dimension $p^2$, so that $\L e_0$ is isomorphic to
$k\widetilde{\mathbb{A}}_{p-1}/J^p$. 

It follows from the above that $\L$ is isomorphic to direct sum of
$k\widetilde{\mathbb{A}}_{p-1}/J^p$ and $p - 1$ copies of $M_p(k)$.
\end{proof}

\section{Enveloping algebra of restricted Lie algebras}

In this section we find enveloping algebras of restricted Lie algebras
occurring in the classification of the finite dimensional connected Hopf
algebras of dimension $p^3$.

First we give a quick review of the definition of the enveloping
algebra of a restricted Lie algebra $L$.  A restricted Lie algebra $L$
is a (finite dimensional) Lie algebra over a field $k$ of
characteristic $p$ with Lie bracket $[-,-]\colon L\times L \to L$ and
a $p$-operation $(-)^{[p]}\colon L\to L$ (for details see
\cite{Jacobson}).  Suppose $L$ has a $k$-basis
$\{x_1,x_2,\ldots, x_n\}$.  By abuse of notation we let
$k\langle x_1,x_2,\ldots,x_n\rangle$ be the free algebra of the
indeterminants $\{x_1,x_2,\ldots,x_n\}$.  Then the universial
enveloping algebra $U^{[p]}(L)$ of $L$ is given by
\[U^{[p]}(L) = k\langle x_1,x_2,\ldots,x_n\rangle/\langle \{x_ix_j - x_jx_i -
  [x_i,x_j]\}_{i,j=1}^{n,n}, \{x_i^p - x_i^{[p]}\}_{i=1}^n\rangle.\]

Next we give the description of the algebras \textbf{C5}, \textbf{C6}
and \textbf{C15} as enveloping algebras of restricted Lie algebras all
of which are $3$-dimensional. 

\begin{prop}
\begin{enumerate}[\rm(a)]
\item Let $L$ be the $3$-dimen\-sional restricted $p$-nilpotent Lie
  algebra with basis $\{x_1,x_2,x_3\}$, the only non-zero bracket
  being $[x_1,x_2] = x_3$ and the $p$-operation given by
  $x_i^{[p]} = 0$ for $i=1,2,3$ for $p>2$.  Then algebra of type
  \textup{\textbf{C5}} is isomorphic to the enveloping algebra of the
  restricted Lie algebra of $L$.
\item Let $L$ be the $3$-dimensional restricted  Lie
  algebra with basis $\{x_1,x_2,x_3\}$ with the only non-zero bracket
  being $[x_1,x_2] = x_3$ and the $p$-operation given by
  $x_1^{[p]} = x_3$ and $x_i^{[p]} = 0$ for $i=2,3$, given in
  \cite[Theorem 2.1 (3/2) (b)]{SU}.  Then algebra of type
  \textup{\textbf{C6}} is isomorphic to the enveloping algebra of the
  restricted Lie algebra of $L$.
\item Let $L$ be the restricted Lie algebra of type
  $\mathfrak{sl}_2(k)$. The algebra of type \textup{\textbf{C15}} is
  isomorphic to the enveloping algebra of the restricted Lie algebra
  of $L$.
\end{enumerate}
\end{prop}
\begin{proof} 
  (a) The algebra of type \textbf{C5} is given by
  \[\L = k\langle x, y,z\rangle/\langle [x,y] - z, [x,z], [y,z], x^p,
  y^p,z^p\rangle.\]  
Let $L$ be the $3$-dimensional Lie algebra with basis
$\{x_1,x_2,x_3\}$ with the only non-zero bracket being $[x_1,x_2] =
x_3$.  We want to show that we can choose a zero $p$-operation to
obtain a restricted Lie algebra.  

We have that 
\[\ad x_1 = [-,x_1] = 
\left(\begin{smallmatrix} 
0 & 0  & 0\\
0 & 0  & 0\\
0 & -1 & 0
 \end{smallmatrix}\right)\quad \text{and}\quad (\ad x_1)^2 = 0,\]
\[\ad x_2 = [-,x_2] = 
\left(\begin{smallmatrix} 
0 & 0 & 0\\
0 & 0 & 0\\
1 & 0 & 0
 \end{smallmatrix}\right)\quad \text{and}\quad (\ad x_2)^2 = 0,\]
\[\ad x_3 = [-,x_3] = 
\left(\begin{smallmatrix} 
0 &   0  & 0\\
0 & 0 & 0\\
0 &   0  & 0
 \end{smallmatrix}\right)\quad \text{and}\quad \ad x_3= 0.
\]
By \cite[Chapter V, Theorem 11]{Jacobson} there is a $p$-operation
$(-)^{[p]}\colon L\to L$ such that $(x_i)^{[p]} = 0$ for $i=1,2,3$.
This shows that $L$ is the $3$-dimensional restricted Lie algebra, and
therefore $\L$ is the enveloping algebra of the $3$-dimensional
restricted Lie algebra $L$.

(b) The algebra of type \textbf{C6} is given by
\[\L = k\langle x, y,z\rangle/\langle [x,y] - z, [x,z], [y,z], x^p - z,
y^p,z^p\rangle\] 
for $p>2$.  Recall the $3$-dimensional restricted  Lie
algebra $L$ with basis $\{x_1,x_2,x_3\}$ with the only non-zero
bracket being $[x_1,x_2] = x_3$ and the $p$-operation given by
$x_1^{[p]} = x_3$ and $x_i^{[p]} = 0$ for $i=2,3$, given in
\cite[Theorem 2.1 (3/2) (b)]{SU}.  Then the algebra $\L$ is isomorphic
to the restricted enveloping algebra of the $3$-dimensional restricted
Lie algebra $L$.

(c) The algebra of type \textbf{C15} is given by
\[\L = k\langle x, y,z\rangle/\langle [x,y] - z, [x,z] - x, [y,z] + y,
x^p, y^p,z^p - z\rangle\] 
for $p>2$.  We shall prove that this is the enveloping algebra of the
$3$-dimensional restricted Lie of type $\mathfrak{sl}_s(k)$.  

Let $L = \mathfrak{sl}_2(k) = ke_+ \oplus kh \oplus ke_-$, where the
standard presentation usually is give as 
\[[h,e_\pm] = \pm 2 e_\pm,\quad [e_+,e_-] = h.\] 
By letting 
\[h' = -\frac{1}{2}h,\quad  e'_+ = \alpha e_+,\quad e'_- = \beta e_-,\]
we obtain the equations 
\[[h',e'_+] = -e_+,\quad [h',e'_-] = e'_-,\quad [e'_+,e'_-] = h',\]
whenever $\alpha\beta = -\frac{1}{2}$ in $k$. 

Now we investigate if there exists a $p$-operation on $L$.  We have that 
\[\ad e'_+ = [-,e'_+] = 
\left(\begin{smallmatrix} 
0 & 0  & -1\\
0 & 0  &  0\\
0 & -1 & 0
 \end{smallmatrix}\right)\quad \text{and}\quad (\ad e'_+)^3 = 0,\]
\[\ad e'_- = [-,e'_-] = 
\left(\begin{smallmatrix} 
0 & 0 & 0\\
0 & 0 & 1\\
1 & 0 & 0
 \end{smallmatrix}\right)\quad \text{and}\quad (\ad e'_-)^3 = 0,\]
\[\ad h' = [-,h'] = 
\left(\begin{smallmatrix} 
1 &   0  & 0\\
0 & -1  & 0\\
0 &   0  & 0
 \end{smallmatrix}\right)\quad \text{and}\quad (\ad h')^p = \left(\begin{smallmatrix} 
1^p &   0  & 0\\
0 & (-1)^p  & 0\\
0 &   0  & 0
 \end{smallmatrix}\right) = \ad h',
\]
since $p$ is odd.  By \cite[Chapter V, Theorem 11]{Jacobson} there is
a $p$-operation $(-)^{[p]}\colon L\to L$ such that
\[ (e'_\pm)^{[p]} = 0\quad \text{and}\quad (h')^{[p]} = h'.\] This
show that $L$ is the $3$-dimensional restricted Lie algebra of type
$\mathfrak{sl}_2(k)$, and therefore $\L$ is the enveloping algebra of
the $3$-dimensional restricted Lie algebra of type $\mathfrak{sl}_2(k)$. 
\end{proof}

\section{Coverings of local algebras}

In this section we find algebras in the classification of the finite
dimensional connected Hopf algebras of dimension $p^3$ which are
coverings of local algebras by a cyclic group.  We start off by
recalling the notion of a covering of a path algebra (for details see
\cite{Gr2}).  

Let $\L = kQ/I$ be an admissible quotient of the path algebra $kQ$
over a field $k$.  Let $G$ be a (finite) group, and let
$w\colon Q_1\to G$ be a \emph{weight function}, that is, just a
function $w\colon Q_1\to G$.  This is extended to a weight function on
all non-trivial paths by defining the \emph{weight of a path}
$p = a_na_{n-1}\cdots a_2a_1$ in $Q$ to be
\[\prod_{i=1}^n w(a_i) = w(a_1)w(a_2)\cdots w(a_{n-1})w(a_n),\]
and for a trivial path $p$ the weight is $w(p) = e$, the identity in
$G$. Assume that the relations in $I$ are homogeneous with respect to
the weight.  With these data we define a covering $\widetilde{Q}(w)$ of
$Q$ with the vertex set 
\[\widetilde{Q}(w)_0 = \{ (v,g)\mid v\in Q_0, g\in G\}\]
and the arrow set 
\[\widetilde{Q}(w)_1 = \{ (a,g)\colon (\frako(a),g)\to (\frakt(a),
  gw(a)) \mid a\in Q_1, g\in G\},\]
where $\frako(a)$ and $\frakt(a)$ denote the origin and the terminus
of the arrow $a$, respectively.  Define a map $\pi\colon
\widetilde{Q}(w)\to Q$ induced from letting 
\[\pi(v,g) = v\quad \text{and}\quad \pi(a,g) = a\]
for all vertices $v$ in $Q_0$, all arrows $a$ in $Q_1$ and all $g$ in
$G$.  It is straightforward to see that given an element $g$ in $G$
and a path $p$ in $Q$, there is a unique path
$\widetilde{p}_g = (p,g)\colon (\frako(p), g) \to (\frakt(p), gw(p))$
in $\widetilde{Q}(w)$ such that $\pi(\widetilde{p}_g) = p$.  This we
extend to a lifting from $kQ$ to $k\widetilde{Q}(w)$.  Hence, given
$g$ in $G$ and a weight homogeneous relation $\sigma$ in $I$, there is
a unique lifting $\widetilde{\sigma}_g$ of $\sigma$ to
$k\widetilde{Q}(w)$.  This is a uniform element, that is, all paths
occuring in $\widetilde{\sigma}_g$ start in the same vertex and end in
the same vertex.  Furthermore, if $I'$ is the ideal generated by the
liftings of a minimal generating set of $I$ and $\widetilde{I}(w)$ is
the set of liftings of elements in $I$, then $I' = \widetilde{I}(w)$
and $\widetilde{I}(w)$ is an ideal in $k\widetilde{Q}(w)$.  In
particular, there is a well-defined algebra map $\pi\colon
k\widetilde{Q}(w)/\widetilde{I}(w) \to kQ/I$.  

We give the following example which will occur in the main result of
this section.
\begin{example}\label{exam:covering}
  Let
  \[\L = k(\xymatrix{1\ar@(ul,dl)[]_y\ar@(ur,dr)[]^z})/\langle y^p, yz
  - zy, z^p\rangle,\] 
 and define $w\colon \{ y, z\}\to C_p = \langle g\rangle$ by letting
$w(y) = g$ and $w(z) = e$.  Let $Q\colon 
\xymatrix{1\ar@(ul,dl)[]_y\ar@(ur,dr)[]^z}$ and $I = \langle y^p, yz
  - zy, z^p\rangle$.  Then the covering
  $k\widetilde{Q}(w)/\widetilde{I}(w)$  is given by the quiver 
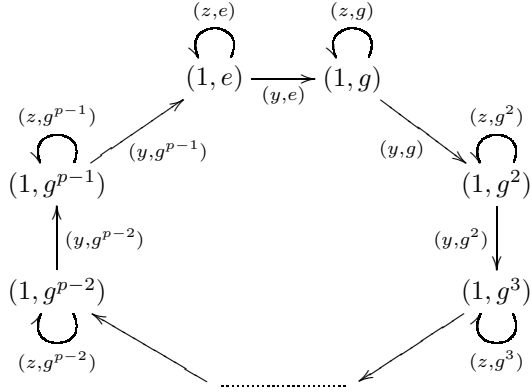
\[\xymatrix{
 & (1, e)\ar[r]_{(y,e)}\ar@(ur,ul)[]_{(z,e)} & (1,g)\ar[dr]_{(y,g)}\ar@(ur,ul)[]_{(z,g)} & \\
(1, g^{p-1})\ar[ur]_{(y, g^{p-1})}\ar@(ur,ul)[]_{(z,g^{p-1})} & & &
(1,g^2)\ar[d]_{(y,g^2)}\ar@(ur,ul)[]_{(z,g^2)} \\ 
(1,g^{p-2})\ar[u]_{(y, g^{p-2})}\ar@(dr,dl)[]^{(z,g^{p-2})} & & & (1,g^3)\ar[dl]\ar@(dr,dl)[]^{(z,g^3)} \\
& \ar[ul]\ar@{..}[r]  & & 
}\]
and relations
\begin{gather}
(y, g^{i+p-1})\cdots (y, g^{i+2})(y, g^{i+1})(y,g^i),\notag\\
(z, g^i)^p,\notag\\
(y,g^i)(z,g^i) - (z,g^{i+1})(y,g^i),\notag
\end{gather} 
for all $i$ in $\{0,1,\ldots, p - 1\}$. 
\end{example}

Now we are ready to give the algebras in the classification of
  finite dimensional connected Hopf algebras of dimension $p^3$ which
  are coverings of a local algebra by a cyclic group.

\begin{prop}
Let $k$ be a field of characteristic $p$.  Let 
\[\L_1 = k\langle y, z\rangle/\langle y^p, yz - zy, z^p\rangle\quad
  \text{and}\quad \L_2 = k\langle 
y, z\rangle/\langle y^p, yz - zy - y^2, z^p\rangle.\]  
Consider the group $G = \langle\sigma\rangle = C_p$. 
\begin{enumerate}[\rm(a)]
\item The algebras of type \textup{\textbf{B1}} and \textup{\textbf{C11}} are
  isomorphic, and they are isomorphic to a covering of the local
  algebra $\L_1$ with the weight function $w\colon \{y,z\}\to G$ given
  by $w(y) = \sigma$ and $w(z) = e$.   
\item The algebra of type \textup{\textbf{B3}} is a covering of the
  local algebra $\L_2$ with the weight function
  $w(y) = w(z) = \sigma$.
\item Assume that $k$ contains the field of order $p^2$ if $p>2$.   The
  algebra of type \textup{\textbf{C16}} is a covering of the local
  algebra $\L_1$ with a weight function $w$ of the form
  $w(y) = \sigma^{-1}$ and $w(z) = \sigma^{-a}$ for some $a$
  (precisely, $a$ such that $\lambda^{-1} = a\lambda$ when $\lambda$
  is the defining parameter occuring in an algebra of type
  \textup{\textbf{C16}}). 
\end{enumerate}
\end{prop}
\begin{proof}
  (a) The algebra of type \textbf{B1} and \textbf{C11} are given as
  \[\L = k\langle x, y, z\rangle/\langle [x,y] - y, [x,z], [y,z], x^p -
  x, y^p, z^p\rangle.\] 
Any element in $\L$ can be written as a linear combination of elements
in the set $\{ x^ry^sz^t\}_{r,s,t=0}^{p-1,p-1,p-1}$.  Let
$\mathfrak{a} = \langle y, z\rangle$ in $\L$.
Since the relations are homogeneous in $y$ and $z$, it follows from
the above that $\mathfrak{a}$ is a nilpotent ideal in $\L$.
Furthermore it is easy to see that
$\L/\mathfrak{a} \simeq k[x]/\langle x^p - x\rangle \simeq k^p$.  We
infer that $\rad\L = \mathfrak{a}$.  By the above comments we have
that $\rad^2\L$ has basis
$\{ x^ry^sz^t\}_{r=0, s+t\geqslant 2}^{p-1, p-1,p-1}$.  The
$\rad\L/\rad^2\L$ has a basis given by the residue classes of the
elements $\{ x^ry, x^rz\}_{r=0}^{p-1}$.  Hence $\L$ is isomorphic to a
quotient of a path algebra $kQ$ over $k$, where $Q$ has $p$ vertices
\[\left\{ v_\alpha = \frac{\prod_{\beta\in \mathbb{Z}_p\setminus
      \{\alpha\}} (x-\beta)}{\prod_{\beta\in \mathbb{Z}_p\setminus
      \{\alpha\}} (\alpha-\beta)}\right\}_{\alpha\in\mathbb{Z}_p}\]
and $2p$ arrows spanned by $\{ x^ry, x^rz\}_{r=0}^{p-1}$.  The linear
span of $\{x^i\}_{i\in \mathbb{Z}_p}$ is equal to the linear span of
$\{v_\alpha\}_{\alpha \in \mathbb{Z}_p}$.  Hence, a basis for
$\rad\L/\rad^2\L$ is $\{v_\alpha y, v_\alpha
z\}_{\alpha\in\mathbb{Z}_p}$.  Since $z$ is in $Z(\L)$, we
have that $v_\alpha z = zv_\alpha$.  
By Remark \ref{rem:idempotents} (2) the element $v_\alpha y$ give rise to an arrow from
vertex $v_{\alpha}\to v_{\alpha + 1}$.   Hence the quiver $Q$ of $\L$ is  
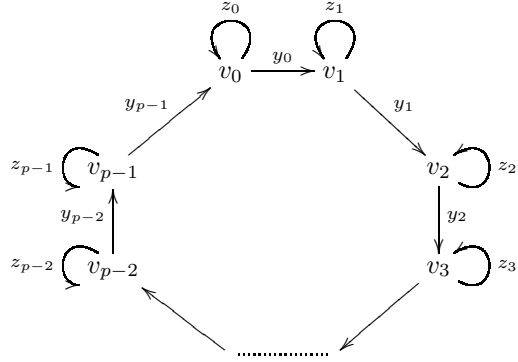
\[\xymatrix{
 & v_0\ar[r]^{y_0}\ar@(ur,ul)[]_{z_0} & v_1\ar[dr]^{y_1}\ar@(ur,ul)[]_{z_1} & \\
v_{p-1}\ar[ur]^{y_{p-1}}\ar@(ul,dl)[]_{z_{p-1}} & & & v_2\ar[d]^{y_2}\ar@(dr,ur)[]_{z_2} \\
v_{p-2}\ar[u]^{y_{p-2}}\ar@(ul,dl)[]_{z_{p-2}} & & & v_3\ar[dl]\ar@(dr,ur)[]_{z_3} \\
& \ar[ul]\ar@{..}[r]  & & 
}\]
The relations are $\{ y_{i+p-1}\cdots y_{i+2}y_{i+1}y_i, z_i^p, y_iz_i -
z_{i+1}y_i\}_{i=0}^{p-1}$.  Hence 
\[\L\simeq kQ/\langle \{y_{i+p-1}\cdots y_{i+2}y_{i+1}y_i, z_i^p,
  y_iz_i - z_{i+1}y_i\}_{i=0}^{p-1}\rangle.\] 
Therefore $\L$ is isomorphic to the covering of the local algebra
$\L_1$ with the weight function
$w\colon \{y,z\}\to C_p = \langle\sigma\rangle$ given by
$w(y) = \sigma$ and $w(z) = e$ (see Example \ref{exam:covering}).

(b) The algebra of type \textbf{B3} is given by
\[\L = k\langle x, y, z\rangle/\langle [x,y]-y,[x,z]-z, [y,z]-y^2,
x^p-x, y^p, z^p\rangle.\]  
The relations $[x,y]-y$, $[x,z]-z$ and $[y,z]-y^2$ imply that $x$'s
can be moved across $y$'s, $x$'s can be moved across $z$'s and that
$z$'s can be moved across $y$'s.  From this we obtain that
$\langle y, z\rangle$ is a nilpotent ideal in
$\L$.  It is easy to see that
$\L/\langle y,z\rangle \simeq k[x]/\langle
x^p-x\rangle\simeq k^p$.  Hence
$\langle y, z\rangle$ is the radical in $\L$,
and $\L\simeq kQ/I$, where $Q$ is a quiver with $p$ vertices given by
\[\left\{ v_\alpha = \frac{\prod_{\beta\in \mathbb{Z}_p\setminus
  \{\alpha\}} (x-\beta)}{\prod_{\beta\in \mathbb{Z}_p\setminus
  \{\alpha\}} (\alpha-\beta)}\right\}_{\alpha\in\mathbb{Z}_p}.\] 
The arrows are given by  $v_\alpha yv_\beta$ and $v_\alpha zv_\beta$ which are non-zero
when $\alpha$ and $\beta$ runs through $\mathbb{Z}_p$.  
Using Remark \ref{rem:idempotents} we get that $v_\alpha y = y v_{\alpha -1}$ and $v_\alpha z = z
v_{\alpha -1}$ for all $\alpha$ in $\mathbb{Z}_p$.  Hence the quiver
is
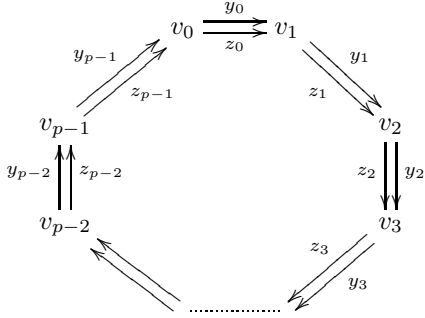
\[\xymatrix{
 & v_0\ar@<0.5ex>[r]^{y_0}\ar@<-0.5ex>[r]_{z_0} & v_1\ar@<0.5ex>[dr]^{y_1}\ar@<-0.5ex>[dr]_{z_1} & \\
v_{p-1}\ar@<0.5ex>[ur]^{y_{p-1}}\ar@<-0.5ex>[ur]_{z_{p-1}} & & & v_2\ar@<0.5ex>[d]^{y_2}\ar@<-0.5ex>[d]_{z_2} \\
v_{p-2}\ar@<0.5ex>[u]^{y_{p-2}}\ar@<-0.5ex>[u]_{z_{p-2}} & & & v_3\ar@<0.5ex>[dl]^{y_3}\ar@<-0.5ex>[dl]_{z_3} \\
& \ar@<0.5ex>[ul]\ar@<-0.5ex>[ul]\ar@{..}[r]  & & 
}\]
The relations are
\[\{y_{i+1}z_i - z_{i+1}y_i - y_{i+1}y_i, y_{i+p-1}y_{i+p-2}\cdots
y_{i+1}y_i, z_{i+p-1}z_{i+p-2}\cdots
z_{i+1}z_i\}_{i\in\mathbb{Z}_p}.\]
Therefore $\L$ is a covering of the local algebra $\L_2$ with the
weight function $w(y) = w(z) = \sigma$.

(c) The algebra of type \textbf{C16} is given as
\[\L = C(\lambda,\delta) = k\langle x, y,z\rangle/\langle [x,y], [x,z] -
\lambda x, [y,z] - \lambda^{-1}y, x^p, y^p, z^p - \delta z\rangle\] 
for some $\lambda\in k^\times$ such that
$\delta = \lambda^{p-1} = \pm 1$.  Any element in $\L$ can be written
as a linear combination of elements in the set
$\{ x^ry^sz^t\}_{r,s,t=0}^{p-1,p-1,p-1}$ of $p^3$ elements, hence it
is a $k$-basis for $\L$.  Let
$\mathfrak{a} = \langle x, y\rangle$ in $\L$.
Since the relations are homogeneous in $x$ and $y$, it follows from
the above that $\mathfrak{a}$ is a nilpotent ideal in $\L$.
Furthermore it is easy to see that
$\L/\mathfrak{a} \simeq k[z]/\langle z^p - \delta z\rangle \simeq k^p$
(note that $\lambda$ lies in the field of order $p^2$ if $\delta =
-1\neq 1$).  We infer 
that $\rad\L = \mathfrak{a}$.  By the above comments we have that
$\rad\L$ has basis
$\{ x^ry^sz^t\}_{r+s\geqslant 1, t=0}^{p-1, p-1,p-1}$ (dimension
$p^3-p$) and $\rad^2\L$ has basis
$\{ x^ry^sz^t\}_{r+s\geqslant 2, t=0}^{p-1, p-1,p-1}$ (dimension
$p^3-3p$).  The vector space $\rad\L/\rad^2\L$ has a basis given by
the residue classes of the elements $\{xz^t, yz^t\}_{t=0}^{p-1}$.

Consider the map $\varphi\colon k[u]\to \L$ given by $\varphi(u) =
z$.  Then $\Ker \varphi = \langle u^p - \delta u\rangle$,
so that $k[z]/\langle z^p - \delta z\rangle$ is a subalgebra of $\L$.
This is a commutative semisimple algebra, since the polynomial $z^p -
\delta z$ is separable. 

Note that the set of roots of $z^p-\delta z$ is closed under addition
and that $0$ is a root, so it is an additive subgroup of order $p$ of
the field.  We take $\lambda$ as in the definition of the algebra.
This is a non-zero root of the polynomial $z^p-\delta z$, and with
this, the set of roots is precisely
\[\{ 0, \lambda, 2\lambda, \ldots, (p-1)\lambda\}.\] 
Since $\lambda^{-1}$ also is a root, there is a unique integer $a$
with $1\leq a\leq p-1$ such that $\lambda^{-1} = a\lambda$.
 
Next we find a set of complete orthogonal idempotent in $\L$.  We have
fixed the roots of $z^p-\delta z$ above.  With this, we have a set of
complete orthogonal idempotents $\{e_0, e_1, \ldots, e_{p-1}\}$ of
$k[z]/\langle z^p - \delta z\rangle$ such that
\[ze_i= i\lambda\cdot e_i.\] 
(for a formula take a variation of Remark
\ref{rem:idempotents}, but we don't need it).  So $\L = 
\oplus_{i=0}^{p-1} \L e_i$ as a left module. 

The set $\{x^uy^ve_i\}_{u,v=0}^{p-1,p-1}$ a $k$-basis for $\L e_i$ (recall
that $x, y$ commute).  In particular $xe_i$ and $ye_i$ are both
non-zero and they generate the radical of $\L e_i$.

We claim that $xe_i$ and $ye_i$ are eigenvectors for left
multiplication with $z$: To see this, we use the relations
\[zx = xz  - \lambda x\quad\text{and}\quad zy = yz  -  \lambda^{-1}y.\]
Therefore  
\[zxe_i = (xz - \lambda x)e_i =  x(i\lambda e_i)   - \lambda(xe_i) =
(i - 1)\lambda (xe_i)\] 
and the eigenvalue is $(i -1)\lambda$.  This implies that $xe_i =
e_{i -1}xe_i$, by the Remark \ref{rem:idempotents}. 
Similarly
\[zye_i = (yz -  \lambda^{-1}y)e_i = (i\lambda - a\lambda)ye_i\]
and the eigenvalue is $(i - a)\lambda$. So $ye_i = e_{i - a}ye_i$. 

As we have found a basis for the radical modulo the radical square, we
can find the quiver $Q$ of $\L$.  We take arrows $\alpha_i=xe_i$ and
$\beta_i=ye_i$ starting at $i$.  By the above, $\alpha_i$ ends at
vertex $i - 1$ and $\beta_i$ ends at vertex $i - a$, where $\lambda^{-1} =
a\lambda$.  It follows from the above that
$\alpha_{i - (p-1)}\cdots\alpha_{i -1}\alpha_i$, $\beta_{i - (p-1)a}\cdots
\beta_{i - a}\beta_i$ and $\beta_{i -1}\alpha_i - \alpha_{i - a}\beta_i$
are relations.  It is easy to see that $kQ/\langle \{\alpha_{i - (p-1)}\cdots\alpha_{i -1}\alpha_i, \beta_{i  - (p-1)a}\cdots
\beta_{i - a}\beta_i,\beta_{i  -1}\alpha_i -
\alpha_{i - a}\beta_i\}_{i=0}^{p-1}\rangle$ has dimension $p^3$, hence
\[\L\simeq kQ/\langle \{\alpha_{i  -(p-1)}\cdots\alpha_{i - 1}\alpha_i, \beta_{i - (p-1)a}\cdots
\beta_{i  -a}\beta_i,\beta_{i -1}\alpha_i -
\alpha_{i  -a}\beta_i\}_{i=0}^{p-1}\rangle.\]
Therefore $\L$ is isomorphic to the covering of the local algebra
$\L_1$ with weight function $w$ of the form $w(y) = \sigma^{-1}$ and
$w(z) = \sigma^{-a}$ for $C_p = \langle\sigma\rangle$.
\end{proof}

\section{Other local algebras}

In this section we discuss the algebras which we cannot classify as we
have done with all the other algebras.  The only class of algebras
left are the algebras of type \textbf{A5}.  Here it is only in
characteristic $2$ that we can identify this algebra.

\begin{prop}
For $p = 2$ the algebra of type \textup{\textbf{A5}} is isomorphic to
the semidihedral algebra of dimension $8$ labelled as \textup{Alg III.1
  (d)} in \textup{\cite[page 298]{E-1428}}. 
\end{prop}
\begin{proof}
For $p=2$ the algebra of type \textbf{A5} is given as
\[\L = k\langle x,y,z\rangle / \langle x^2, y^2, [x,y], [x,z], [y, z] -
x, z^2 + xy\rangle.\]
We observe that $x$ is central, and also is not a generator. We substitute $x=[y,z]$,  this is central, and the
other  
relations translate to 
\[ [y,z]^2, y^2,  z^2 + [y,z]y.
\]
Noting that char$(k)=2$, the last two relations mean that in the algebra $y^2=0$ and $z^2 = yzy$. With these, the first relation is
equivalent to $(yz)^2 = (zy)^2$ in the algebra. Moreover if $y^2=0$ and $z^2 = yzy$  then it follows
that $[y,z]$ is central.  We get that  
\[\L \simeq k\langle y,z\rangle/\langle (yz)^2  - (zy)^2, y^2, z^2 -
yzy\rangle,\]
and these relations imply that $(zy)^2z = 0$. 

The radical layers of this algebra have basis $\{1\}$, $\{y, z\}$,
$\{yz, zy\}$, $\{yzy, zyz\},$ $\{yzyz\}$, hence of dimensions $\{1, 2,
2, 2, 1\}$ and structure of the projective is
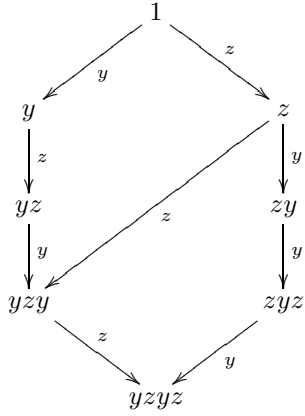
\[\xymatrix{
& 1\ar[dl]^y\ar[dr]^z & \\
y\ar[d]^z & & z\ar[ddll]^z\ar[d]^y\\
yz\ar[d]^y & & zy\ar[d]^y\\
yzy\ar[dr]^z & & zyz\ar[dl]^y\\
& yzyz &}\]
This algebra is isomorphic to a semidihedral algebra of dimension $8$
labelled as Alg III.1 (d) in \cite[page 298]{E-1428}.
\end{proof}

\begin{remark}
  The algebra $\L$ of type \textbf{A5} for $p = 2$ is a local algebra
  and not semisimple.  It is not a selfinjective Nakayama algebra,
  since the radical of the indecomposable projective is generated by
  two elements.  We don't know if it is an enveloping algebra of a
  restricted Lie algebra.  It is not a non-trivial covering of any
  other algebra, since it is local.

  We also claim that it is not a group algebra as we argue as
  follows. If $\L$ is isomorphic to a group algebra $kG$, then
  $|G| = 8$.  The algebra is not commutative, and the only non-abelian
  groups of order $8$ are the dihedral group $D_8$ and the quaternion
  group $Q_8$.  If $G = D_8$, then $kG$ is isomorphic to the algebra
  labelled III.1 (c) in \cite[page 294]{E-1428}, which is not isomorhpic to
  $\L$.  If $G = Q_8$, then $kG$ is isomorphic to the algebra labelled
  III.1 (e) or (e'), which again is not isomorphic to $\L$.
\end{remark}

\begin{remark}
For $p > 2$ the algebra of type \textbf{A5} is given as
\[\L(\beta) = k\langle x,y,z\rangle/\langle x^p, y^p, [x,y], [x,z], [y,
z] - x, z^p + x^{p - 1}y - \beta x\rangle.\]
We can see that
the ideal $\langle x, y, z\rangle$ in $\L$ is the radical of $\L$.
Furthermore, $\L$ is generated as an algebra by $\{y,z\}$, and direct
computations show that
\[\L \simeq k\langle y,z\rangle/\langle [y,z]^p, y^p, 2yzy -
zy^2 -y^2z, yz^2 + z^2y - 2zyz, z^p + (yz-zy)^{p-1}y -
\beta(yz-zy)\rangle.\]

Let $\beta \neq 0$.  Then one can show directly that the dimension of
the radical quotients are 
\[1,2, 3,\ldots, p-1, \underbrace{p, p, \ldots, p}_{p^2 - p + 1},
  p-1,\ldots, 3,2,1.\]
If $\L(\beta) \simeq kG$, then $G$ is a non-abelian group of order
$p^3$.  It is well-known that up to isomorphism there are two
non-abelian groups of order $p^3$, namely $C_{p^2}\rtimes C_p$ and
$(C_p\times C_p)\rtimes C_p$.  The theorem of Jennings determines explicitly
in terms of the group data a basis of $kG$ for $G$ a group, compatible
with the radical series.  

The group algebra of $G = C_{p^2}\rtimes C_p$ has the same dimensions of the
radical quotients as $\L(\beta)$.  If there exists an isomorphism
$\psi\colon \L(\beta)\to kG$ the element $\psi([y,z])$ must be in the
center.  One can show using a presentation of $kG$ that this is not
possible. 

Let $\beta = 0$.  Then the algebra $\L(0)$ has the same radical
quotients as the group algebra $k((C_p\times C_p)\rtimes C_p)$ and similarly
the algebras are not isomorphic.
\end{remark}

\section{Eckmann-Shapiro Lemma and $\Ext$-algebras}
This section is devoted to studying a situation for an adjoint pair of
functors where Noetherianity of $\Ext$-algebras can be transferred
from one category to the other via the first or the second functor in
the adjunction.  This is the Eckmann-Shapiro Lemma in the setting
  of abelian categories.

Let $\C$ and $\D$ be two abelian categories with enough projective
and enough injective objects.  Let $F\colon \C\to \D$ and $G\colon
\D\to\C$ be a pair $(F,G)$ of adjoint functors, and denote the
adjunction by
\[\psi = \psi_{C,D}\colon \Hom_\D(F(C),D) \to \Hom_\C(C,G(D))\]
for all $C$ in $\C$ and $D$ in $\D$.  With this setup we are ready to
formulate the transfer of Noetherianity of $\Ext$-algebras across the
adjunction using the functor $F$. 

\begin{prop}\label{prop:adjointNoetherian}
Let $(F,G)$ be an adjoint pair of functors as above, and assume that
$F$ is exact and preserves projective objects. 
\begin{enumerate}[\rm(a)]
\item The adjunction induces a functorial isomorphism 
\[\Ext^i_\D(F(C),D) \simeq \Ext^i_\C(C,G(D))\]
for all $i\geq 1$ and for all $C$ in $\C$ and $D$ in $\D$.  We also
denote this isomorphism by $\psi$. 
\item For every element $\eta$ in $\Ext^m_\C(C',C)$ and $\theta$ in
  $\Ext^n_\C(C,G(D))$ the following equality holds in
  $\Ext^{m+n}_\D(F(C'),D)$: 
\[\psi^{-1}(\theta)\cdot F(\eta) = \psi^{-1}(\theta\cdot \eta).\]
\item Let $n$ be a non-negative integer.  If
  $\Ext_\C^{\geq n}(C,G(D))$ is a finitely generated right
  $\Ext^*_\C(C,C)$-module, then $\Ext_\D^{\geq n}(F(C),D)$ is a
  finitely generated right $\Ext_\D^*(F(C),F(C))$-module.
\item If $\Ext^*_\C(C,C)$ is a right Noetherian ring and $\Ext^*_\C(C,G(F(C)))$
  is a finitely generated right $\Ext^*_\C(C,C)$-module, then
  $\Ext^*_\D(F(C),F(C))$ is a right Noetherian ring. 
\end{enumerate}
\end{prop}
\begin{proof}
(a) Let $\cdots \to P_1\to P_0\to C\to 0$ be a projective resolution
of $C$ in $\C$.  Then $\cdots \to F(P_1)\to F(P_0)\to F(C)\to 0$ is a
projective resolution of $F(C)$ in $\D$.  Applying the functor
$\Hom_\C(-,G(D))$ to the first projective resolution and
$\Hom_\C(-,D)$ to the second projective resolution we obtain the
following commutative diagram where all vertical maps are
isomorphisms:
\[\xymatrix{
0\ar[r] & \Hom_\C(C,G(D)) \ar[r]\ar[d]^{\psi^{-1}} & \Hom_\C(P_0,G(D))\ar[r]\ar[d]^{\psi^{-1}}
& \Hom_\C(P_1,G(D))\ar[r]\ar[d]^{\psi^{-1}} & \cdots \\
0\ar[r] & \Hom_\D(F(C),D) \ar[r] & \Hom_\D(F(P_0),D)\ar[r]
& \Hom_\D(F(P_1),D)\ar[r] & \cdots }\]
The claim follows immediately from this commutative diagram. 

(b) Let $\eta$ be an element in $\Ext^m_\C(C',C)$ and $\theta$ an
element in $\Ext^n_\C(C,G(D))$.  Let 
\[\mathbb{P}'\colon \cdots \to P_m'\to P_{m-1}'\to \cdots \to P_0'\to
  C'\to 0\]
and 
\[\mathbb{P}\colon \cdots \to P_n\to P_{n-1}\to \cdots \to P_0\to
  C\to 0\] 
be projective resolutions of $C'$ and $C$, respectively.  Here $P_0'$
and $P_0$ are in degree $0$.  Then $\eta$ and $\theta$ can be
represented by maps $P_m'\to C$ and $P_n\to G(D)$, which we also
denote by $\eta$ and $\theta$, respectively.  These give rise to chain
maps
$\widetilde{\eta}\colon \mathbb{P}_{\geq 0}'\to \mathbb{P}_{\geq
  0}[-m]$ and $\theta\colon \mathbb{P}_{\geq 0}\to G(D)[-n]$, viewing
$G(D)$ as a stalk complex concentrated in degree $0$.  Then the Yoneda
products $\theta\eta$ and $\psi^{-1}(\theta)F(\eta)$ can be
represented by the morphisms
$\theta[-m]\widetilde{\eta} = \theta\eta_n$ and
$\psi^{-1}(\theta)F(\eta_n)$, where $\eta_n$ is a $n$-th lifting of
$\eta$.  We have the following commutative diagram of complexes
\[\xymatrix{
\Hom_\D(F(\mathbb{P}_{\geq 0}), D[-n])
\ar[r]^\psi\ar[d]_{F(\widetilde{\eta}[m])^*} & \Hom_\C(\mathbb{P}_{\geq 0}, G(D)[-n]) \ar[d]^{\widetilde{\eta}[m]^*} \\
\Hom_\D(F(\mathbb{P}_{\geq 0}'[m]), D[-n])\ar[r]^\psi&
\Hom_\C(\mathbb{P}_{\geq 0}'[m],  G(D)[-n]) 
}\]
Starting with $\psi^{-1}(\theta)$ in the upper left corner and taking
homology, it follows that the Yoneda products
$\psi^{-1}(\theta)F(\eta)$ and $\psi^{-1}(\theta\eta)$ are equal in
$\Ext^{m+n}_\D(F(C'),D)$.

(c) Let $n$ be a non-negative integer.  Assume that 
  $\Ext_\C^{\geq n}(C,G(D))$ is a finitely generated right
  $\Ext^*_\C(C,C)$-module.  Suppose $\{\theta_1,\ldots,\theta_t\}$ is
  a set of generators of $\Ext_\C^{\geq n}(C,G(D))$ as a
  $\Ext^*_\C(C,C)$-module.  Consider the set $\{\psi^{-1}(\theta_1),
  \ldots, \psi^{-1}(\theta_t)\}$ in $\Ext^{\geq n}_\D(F(C),D)$.  Then
  it follows immediately from (b) that this is a generating set as a
  $\Ext_\D^*(F(C),F(C))$-module. 

  (d) Let $\Sigma = F(\Ext^*_\C(C,C))$. Then $\Sigma$ is a subalgebra
  of $\Ext^*_\D(F(C),F(C))$.  Let $U$ be a right ideal in
  $\Ext^*_\D(F(C),F(C))$. In particular, $U$ is a right
  $\Sigma$-submodule of $\Ext^*_\D(F(C),F(C))$.  This corresponds to a
  right $\Ext^*_\C(C,C)$-submodule $U'$ of $\Ext^*_\C(C,GF(C))$.
  Since $\Ext^*_\C(C,C)$ is Noetherian and $\Ext^*_\C(C,GF(C))$ is a
  finitely generated $\Ext^*_\C(C,C)$-module, the submodule $U'$ is a
  finitely generated $\Ext^*_\C(C,C)$-module.  Then we infer that $U$
  is a finitely generated $\Sigma$-module and therefore a finitely
  generated $\Ext^*_\D(F(C),F(C))$-module, hence a finitely generated
  right ideal.  This shows that $\Ext^*_\D(F(C),F(C))$ is right
  Noetherian. 
\end{proof}
Next we state the dual result, where we leave the proof to the reader.
\begin{prop}
Let $(F,G)$ be an adjoint pair of functors as above, and assume that
$G$ is exact and preserves injective objects. 
\begin{enumerate}[\rm(a)]
\item The adjunction induces a functorial isomorphism 
\[\Ext^i_\D(F(C),D) \simeq \Ext^i_\C(C,G(D))\]
for all $i\geq 1$ and for all $C$ in $\C$ and $D$ in $\D$.  We also
denote this isomorphism by $\psi$. 
\item For every element $\eta$ in $\Ext^m_\D(D,D')$ and $\theta$ in
  $\Ext^n_\D(F(C),D)$ the following equality holds in
  $\Ext^{m+n}_\C(C,G(D'))$: 
\[\psi(\eta\cdot \theta) = G(\eta) \cdot \psi(\theta).\]
\item\sloppypar Let $n$ be a non-negative integer.  If
  $\Ext_\D^{\geq n}(F(C),D)$ is a finitely generated left
  $\Ext^*_\D(D,D)$-module, then $\Ext_\C^{\geq n}(C,G(D))$ is a
  finitely generated left $\Ext_\C^*(G(D),G(D))$-module.
\item If $\Ext^*_\D(D,D)$ is a left Noetherian ring and
  $\Ext^*_\D(F(G(D)), D)$ is a finitely generated left
  $\Ext^*_\D(D,D)$-module, then $\Ext^*_\C(G(D),G(D))$ is a left
  Noetherian ring.
\end{enumerate}
\end{prop}
We end this section by an application of the above.

Let $\nu\colon \G \to \L$ be an algebra inclusion of finite
dimensional $k$-algebras for a field $k$.  Then we have the induction
functor 
\[F = \L\otimes_\G - \colon \mod\G\to\mod\L\]
and the restriction functor 
\[H = \Hom_\L({_\L \L}_\G,-)\colon \mod\L\to\mod\G,\]
which clearly is an adjoint pair $(F,H)$ of functors. 

\begin{cor}\label{cor:algebrainclusion}
  Let $\nu\colon \G \to \L$ be an algebra inclusion of finite
  dimensional $k$-algebras for a field $k$.  Let $\mathfrak{r}$ denote
  the Jacobson radical $\rad \G$ of $\G$, and assume that
  $\L\mathfrak{r} = \rad\L$.  Moreover assume that
  $\L_\G$ is a projective $\G$-module.  

If $\Ext^*_\G(\G/\mathfrak{r},\G/\mathfrak{r})$ is right Noetherian,
then $\Ext^*_\L(\L/\rad\L,\L/\rad\L)$ is right Noetherian. 
\end{cor}
\begin{proof} Note that $F$ is exact and takes projectives to
  projectives since $\L_\G$ is projective.  By assumption
  $\Ext^*_\G(\G/\mathfrak{r},\G/\mathfrak{r})$ is a right Noetherian
  algebra.  Using induction on the length of a module and that
  $\Ext^*_\G(\G/\mathfrak{r},\G/\mathfrak{r})$ is right Noetherian, it
  follows that $\Ext^*_\G(\G/\mathfrak{r}, M)$ is a finitely generated
  right $\Ext^*_\G(\G/\mathfrak{r},\G/\mathfrak{r})$-module for all
  finitely generated $\G$-modules $M$.

  We have $\L\mathfrak{r} = \rad\L$, the functor $F$ is exact and
  $F(\G) = \L\otimes_\G\G \xrightarrow{\mult} \L$ is an
  isomorphism. It follows that $F(\G/\mathfrak{r}) \simeq \L/\rad\L$.
  Then Proposition \ref{prop:adjointNoetherian} (d) implies that
  $\Ext^*_\L(\L/\rad\L,\L/\rad\L)$ is right Noetherian.
\end{proof}

\section{Twisted tensor product algebras}

We recall the notion of a twisted tensor product of two algebras
discussed in \cite{CR}, and the fact that a covering of an
  algebra can be viewed as a twisted tensor product.

Let $\L$ and $\G$ be two algebras over a commutative ring $k$.  Let
$\tau\colon \G\otimes_k \L\to \L\otimes_k \G$ be a linear map. Define
the following operation on $\L\otimes_k \G$ by letting 
\[(\lambda\otimes \gamma)\cdot_\tau (\lambda'\otimes \gamma') = 
\lambda \tau(\gamma \otimes \lambda')\gamma'\]
for $\lambda,\lambda'$ in $\L$ and $\gamma,\gamma'$ in $\G$.
Denote $\L\otimes_k\G$ with this structure by $\L\otimes_\tau \G$.  
If the linear map $\tau$ satisfies the following equalities
\begin{align}
(1\otimes \gamma')\cdot_\tau \tau(\gamma\otimes\lambda) & =
                                                          \tau(\gamma'\gamma\otimes
                                                          \lambda)\notag\\
\tau(\gamma\otimes \lambda)\cdot_\tau (\lambda'\otimes 1) & =
                                                            \tau(\gamma\otimes
                                                            \lambda\lambda')\notag\\
\tau(\gamma\otimes 1) & = 1 \otimes \gamma\notag\\
\tau(1 \otimes \lambda) & = \lambda\otimes 1\notag
\end{align}
for all $\lambda,\lambda'$ in $\L$ and all $\gamma,\gamma'$ in
$\G$, then $\L\otimes_\tau \G$ is an associative
algebra.  It is called a \emph{twisted tensor product of $\L$ and $\G$ over
the twisting $\tau$}.  

Assume from now on that $\L\otimes_\tau\G$ is a twisted tensor product
over the twisting $\tau$.  Then the natural maps $\nu_\L\colon \L\to
\L\otimes_\tau \G$ and $\nu_\G\colon \G\to \L\otimes_\tau \G$ given by
$\lambda\mapsto \lambda\otimes 1$ and $\gamma\mapsto 1\otimes \gamma$,
respectively are homomorphisms of algebras.  Then $\L\otimes_\tau \G$
is a $(\L\otimes_\tau\G)$-$\G$-bimodule, so that we have the induction
and the restriction functors 
\[F = \L\otimes_\tau\G\otimes_\G - \colon \mod\G\to \mod\L\otimes_\tau\G\]
and
\[G = \Hom_{\L\otimes_\tau\G}(\L\otimes_\tau\G,-) \colon
                                                 \mod\L\otimes_\tau\G\to \mod\G\]
respectively.  These functors are an adjoint pair, so we have
functorial isomorphisms 
\[\Hom_{\L\otimes_\tau\G}(\L\otimes_\tau\G\otimes_\G M,N) \simeq 
\Hom_\G(M, \Hom_{\L\otimes_\tau\G}(\L\otimes_\tau\G, N)) \simeq 
\Hom_\G(M, {_\G N}).\]
The functor $F$ preserves projective modules.  Since
  $\L\otimes_\tau\G$ is isomorphic to $\L\otimes_k\G$ as a right
  $\G$-module, it is a projective right $\G$-module if and only if
  $\L$ is a projective $k$-module.  Assume that $\L$ is a
  projective $k$-module.  
Then $F$ is an exact functor and if
$\mathbb{P} \to M$ is a projective resolution of $M$ as a $\G$-module,
then $F(\mathbb{P})\to F(M)$ is a projective resolution of $F(M)$ as a
$\L\otimes_\tau\G$-module.  Using the above adjunction we infer that 
\begin{align}
\Ext^i_{\L\otimes_\tau\G}(\L\otimes_\tau\G\otimes_\G M,N)  & \simeq
\Ext^i_\G(M,\Hom_{\L\otimes_\tau\G}(\L\otimes_\tau\G,N))\notag\\
& \simeq
  \Ext^i_\G(M, {_\G N})\notag 
\end{align}

We shall apply twisted tensor products to coverings of algebras.  To
this end we need to recall the dual of a group algebra.  Let $G$ be a
(finite) group, then
\[k[G]^* = \{ \sum_{g\in G} \alpha_g \epsilon_g\mid \text{almost all\
  } \alpha_g \text{\ are zero}\},\] 
where $\epsilon_g\epsilon_h = \delta_{g,h}\epsilon_g$.  With these
preliminaries we can state one of the main results from \cite{CR},
namely \cite[Theorem 5.7]{CR}, taking into account the final comment
in the paper concerning the extension of that result to coverings of
quivers with relations.

\begin{thm}[\protect{\cite[Theorem 5.7]{CR}}]
Let $kQ/I$ be an admissible quotient of the path algebra $kQ$.  Let
$w\colon Q_1\to G$ be a weight function into a (finite) group, and
assume that $I$ is generated by weight homogeneous elements.  Then the
covering $k\widetilde{Q}(w)/\widetilde{I}(w)$ is isomorphic to
$k[G]^*\otimes_\tau kQ/I$, where the twisting map $\tau\colon
kQ/I\otimes_k k[G]^*\to k[G]^*\otimes_k kQ/I$ is given by
$\tau(\overline{p}\otimes \epsilon_{g}) =
\epsilon_{gw(\overline{p})}\otimes \overline{p}$. 
\end{thm}
For an illustration see Example \ref{exam:covering}. 

Using the general setup with adjoint pair of functors for a twisted
tensor product of algebras reviewed above, we apply Proposition
\ref{prop:adjointNoetherian} to coverings of quivers with relations by
a finite group. 

\begin{prop}\label{prop:fingencoverings}
  Let $\G = kQ/I$ be an admissible quotient of the path algebra $kQ$.
  Let $w\colon Q_1\to G$ be a weight function into a finite group,
  and assume that $I$ is generated by weight homogeneous elements.
  Assume that $\Ext^*_\G(\G/\mathfrak{r},\G/\mathfrak{r})$ is a right Noetherian
  algebra.  Let $\L = k\widetilde{Q}(w)/\widetilde{I}(w)$.
  \begin{enumerate}[\rm(a)]
  \item Let $T$ be the direct sum of all isomorphism classes of simple
    $\L$-modules. Then
    $\Ext^*_{\L}(T, T)$ is a right Noetherian algebra.
\item For each simple $\L$-module $S$ the algebra
  $\Ext^*_{\L}(S,S)$ is right Noetherian. 
\end{enumerate}
\end{prop}
\begin{proof}
  (a) We have that $\L \simeq k[G]^*\otimes_\tau \G$
  as defined above. Then we have an inclusion of finite
  dimensional $k$-algebras $\nu\colon \G\to \L$.  The right
  $\G$-module $\L$ is isomorphic to $\G^{|G|}_\G$, hence free and
  in particular projective.  By the definition of the twisting $\tau$,
  the subset $k[G]^*\otimes_\tau\mathfrak{r}$ is an ideal in $\L$ and
  in fact a nilpotent ideal in $\L$.  Furthermore, since the weight of
  the trivial paths is the identity in $G$, it is straightforward to
  see that
\[(k[G]^*\otimes_\tau \G)/(k[G]^*\otimes_\tau \mathfrak{r}) \simeq
  k[G]^*\otimes_\tau \G/\mathfrak{r} \simeq k[G]^*\otimes_k
  \G/\mathfrak{r},\]
where in the last isomorphic we use that the weight of any vertex is
the identity in $G$.  The last algebra in the above formula is clearly
semisimple.  It follows that  
\[\rad\L = k[G]^*\otimes_\tau\mathfrak{r} = (k[G]^*\otimes_k
  \G)(1\otimes\mathfrak{r})
  \]
Now we can apply Corollary \ref{cor:algebrainclusion} to obtain (a). 

(b) Let $T = \oplus_{i=1}^n S_i$, where $S_i$ is a simple module.  Let
$f = f_1$ be the idempotent in $\Hom_{k[G]^*\otimes_\tau\G}(T,T)$
corresponding to the direct summand $S=S_1$ of $T$.  Then
$P = f\Ext^*_{k[G]^*\otimes_\tau\G}(T,T)$ is a finitely generated
projective right $\Ext^*_{k[G]^*\otimes_\tau\G}(T,T)$-module.
Since $\Ext^*_{k[G]^*\otimes_\tau\G}(T,T)$ is a right Noetherian
algebra, it follows from \cite[Proposition 2.3 (i)]{S} (see also
\cite{Ha,F}) that
  \[\End_R(P) \simeq f \Ext^*_{k[G]^*\otimes_\tau\G}(T,T)f
    \simeq \Ext^*_{k[G]^*\otimes_\tau\G}(S,S)\] 
  is right Noetherian. This completes the proof of the proposition. 
\end{proof}

\section{Cohomology rings}

In this final section we show that the cohomology
ring of the various algebras occurring the in the classification of
the finite dimensional connected Hopf algebras of dimension $p^3$
are Noetherian, except for the algebras of type \textbf{A5}
for $p > 2$.

Before reviewing the cohomology rings of the algebras in the
classification we recall the following general result on cohomology of
tensor products of algebras. 

Let $\L$ and $\G$ be two (finite dimensional) $k$-algebras for a field
$k$.  Then we can form their tensor product $\L\otimes_k\G$ with
componentwise multiplication.  Given a $\L$-module $M$ and a
$\Gamma$-module $N$, their tensor product $M\otimes_k N$ over $k$ is
in a natural way a module over $\L\otimes_k\G$.  Furthermore it is
well-known that 
\[\Ext^*_{\L\otimes_k\G}(M\otimes_k N,M\otimes_k N) \simeq
\Ext^*_\L(M,M)\overline{\otimes}_k \Ext^*_\G(N,N),\]
where $\overline{\otimes}_k$ denotes the usual tensor product with the
product of homogeneous elements of degrees $d_1$ and $d_2$ is twisted
by $(-1)^{d_1d_2}$ (that is, odd degree elements anti-commute).  In
particular, if $\L$ and $\G$ are augmented $k$-algebras, then $k$ is
in a natural way a $\L$-module and a $\G$-module, so that we can
consider the cohomology rings $\Ext^*_\L(k,k)$ and $\Ext^*_\G(k,k)$.
Then it follows from the above that 
\begin{equation}\label{eq:kunneth}
\Ext^*_{\L\otimes_k\G}(k,k) \simeq
\Ext^*_\L(k,k)\overline{\otimes}_k\Ext^*_\G(k,k).
\end{equation}

When $k$ is a field of characteristic $p$, then
$kC_p \simeq k[x]/\langle x^p\rangle$.  Hence it is of interest to
discuss the cohomology ring of algebras of the form
$k[x]/\langle x^n\rangle$ for $n\geq 2$.  When
$\L_n=k[x]/\langle x^n\rangle$ for $n\geqslant 2$, then we have (see
\cite[Lemma 5.2]{BO})
\begin{equation}\label{eq:truncpoly}
\Ext^*_{\L_n}(k,k) \simeq \begin{cases}
k[y], & \text{when $n=2$,}\\
k[y,z]/\langle y^2\rangle, & \text{when $n>2$}
\end{cases}
\end{equation}
where the degree of $y$ is $1$ and the degree of $z$ is $2$. 

\subsection{Semisimple algebras}

Let $H$ be a semisimple Hopf algebra.  Since the global dimension of a
semisimple algebra is zero, it is well-known that cohomology algebra
of $H$ is isomorphic to $k$.  This gives the cohomology
ring of the only semisimple algebra in the classification, namely the
case \textbf{C1}.  The cohomology ring is clearly Noetherian.

\subsection{Group algebras} 

The algebras of type \textbf{A3} = \textbf{C4}, \textbf{C7},
\textbf{C9} and \textbf{C10} are elementary abelian groups, $(C_p)^r$
for some $r\geq 1$.  For a field $k$ of characteristic $p$ and
$k(C_p)^r$ it is known that
\begin{equation}\label{eq:groupkunneth}
H^*((C_p)^r,k) = \Ext^*_{k(C_p)^r}(k,k) =
  \wedge(x_1,x_2,\ldots,x_r)\otimes_k k[y_1,y_2,\ldots,y_r]
\end{equation}
where the degree of $x_i$ is $1$ and the degree of $y_i$ is $2$.
Hence this describes the cohomology rings of the above mentioned
algebras. 

The algebras of type \textbf{A2} = \textbf{C2} and \textbf{C8} are
group algebras of cyclic groups of order $p^3$ and $p^2$,
respectively.  Since the field $k$ has characteristic $p$ they are
isomorphic to $k[x]/\langle x^{p^3}\rangle$ and $k[x]/\langle
x^{p^2}\rangle$.  The structure of the cohomology ring of these
algebras are then described by \eqref{eq:truncpoly}. 

For the algebra of type \textbf{A4} = \textbf{C3} we have to use both
\eqref{eq:kunneth} and \eqref{eq:truncpoly} to describe the cohomology
ring.

In all cases the cohomology ring of the algebra is Noetherian. 

\subsection{Selfinjective Nakayama algebras}
Let $\L_{n,t} = k\widetilde{\mathbb{A}}_{n-1}/J^t$ for a field $k$,
which is a selfinjective with $n$ indecomposable projective modules of
length $t$.  The $\Ext$-algebra of a simple module over $\L_{1,t}$ is
described in \eqref{eq:truncpoly}.  So we assume that $n > 1$. 

\begin{lem}
  Let $S$ be a simple module over $\L_{n,t}$, and $t = qn + r$ for
  $q \geq 0$ and $0 \leq r < n$ with $n > 1$.  Define $l_1$ to be the
  order of the subgroup $\langle r + \mathbb{Z}n\rangle$ of
  $\mathbb{Z}_n$, and define $l_2 \geq 0$ smallest possible such that
  $l_2r +1 \equiv 0 \mod n$, if such an $l_2$ exists.  Then the
  $\Ext$-algebra of $S$ is given as
\[\Ext^*_{\L_{n,t}}(S,S) \simeq 
\begin{cases}
k[x], \text{\ if $l_2$ doesn't exist},\\
k[x,y]/\langle y^2\rangle, \text{\ if $l_2$ exists,}
\end{cases}\]
where $x$ has degree $2l_1$ and $y$ has degree $2l_2 + 1$. 
\end{lem}
\begin{proof}
Each indecomposable module over $\L_{n,t}$ can be identitified with
the radical layers, for instance, the projective associated to vertex
$1$ is 
\[(\underbrace{1,2,\ldots,n,1,2,\ldots,n,\ldots,1,2,\ldots,n}_q,1,2,\ldots,
  r),\] where $1$ symbolize the simple module associated to vertex
$1$, and so on.  Or for short $((1,2,\ldots,n)^q,1,2,\ldots,r)$.  Then
it is easy to see that the top of the $2l$-th syzygy is $lr + 1$ and
the top of the $(2l + 1)$-th syzygy is $lr + 2$.  These numbers are
always taken module $n$, where we choose $\{ 1,2,\ldots, n\}$ as
representatives for the different equivalence classes.  The smallest
$l\geq 1$ such that $lr +1 \equiv 1 \mod n$ is the order of the
subgroup $\langle r + \mathbb{Z}n\rangle$ of $\mathbb{Z}_n$, that is,
the top of the $2l$-th syzygy is $1$.  This means that
$\Ext^{2l_1r}_{\L_{n,t}}(S,S) \simeq k$ for all $r\geq 1$.  If there
exists a smallest $l_2$ such that $l_2r + 2\equiv 1 \mod n$, that is,
the top of the $(2l + 1)$-th syzygy is $1$.  This means that
$\Ext^{2(l_2+l_1r) + 1}_{\L_{n,t}}(S,S) \simeq k$ for all $r\geq 1$.
Direct computations show that the structure of the $\Ext$-algebra of
the simple module is as claimed.
\end{proof}
Using this we find that the cohomology ring of the algebra of type
\textbf{B2} is $k[x]$ with the degree of $x$ equal to $2p$, of the
algebra of type \textbf{C12} is $k[x]$ with the degree of $x$ equal to
$2$, of the algebra of type \textbf{C13} is $p$ copies of $k[x]$ with
the degree of $x$ equal to $2$ and of the algebra of type \textbf{C14}
is $k[x]\oplus (M_p(k))^{p-1}$ with the degree of $x$ equal to $2$. 

In all cases the cohomology ring of the algebra is Noetherian.

\subsection{Enveloping algebras of restricted Lie algebras}

The enveloping algebras of restricted Lie algebras are finite
dimensional cocommutative Hopf algebras, and for such algebras the
cohomology ring is finitely generated by \cite[Theorem 1.1]{FS}.  This
shows that the cohomology rings of the algebras of
type \textbf{C5}, \textbf{C6} and \textbf{C15} are Noetherian.

\subsection{Coverings of local algebras} 
Let $k$ be a field of characteristic $p$.  Let 
\[\L_1 = k\langle y, z\rangle/\langle y^p, yz - zy, z^p\rangle\quad
  \text{and}\quad \L_2 = k\langle 
y, z\rangle/\langle y^p, yz - zy - y^2, z^p\rangle.\]  
The algebra $\L_1$ is isomorphic to $k[y]/\langle y^p\rangle\otimes_k
k[z]/\langle z^p\rangle\simeq k(C_p\times C_p)$, so that using \eqref{eq:truncpoly} and
\eqref{eq:kunneth} or \eqref{eq:groupkunneth} one can describe the
structure of the cohomology ring of this algebra and it is a
Noetherian algebra.   

For $p = 2$ the algebra $\L_2$ is given by
$\L_2 = k\langle y,z\rangle/\langle y^2, yz-zy, z^2\rangle$, which is
a Koszul algebra.  Then
$\Ext^*_{\L_2}(k,k)\simeq k\langle y,z\rangle/\langle yz + zy\rangle$,
which is Noetherian.  For $p > 2$ the cohomology ring of $\L_2$ is
shown to be Noetherian in \cite[Section 4]{NW}.  

The algebras $H$ of type \textbf{B1} = \textbf{C11} and type
\textbf{C16} are coverings of the algebra $\L_1$ and the algebras of
type \textbf{B3} are coverings of $\L_2$.  Then by Proposition
\ref{prop:fingencoverings} (b) it follows that cohomology ring
$S = \Ext^*_H(k,k)$ of $H$ is right Noetherian.  Since $S$ is graded
commutative, it follows that $S$ is Noetherian. 

\subsection{Other local algebras} 
The cohomology ring of the algebra of type \textbf{A5} is to our
knowledge unknown in general.  When $p = 2$, it is a semidihedral
algebra of dimension $8$.  This algebra is the smallest algebra in a
family which contains group algebras of semidiheral $2$-groups.  The
cohomology of semidihedral $2$-groups is known and a presentation can
be found in \cite{Sasaki}.  The cohomology of the algebra of type
\textbf{A5} for $p = 2$ is closely related.

\subsection{Conclusion} Our investigations show that there is only one
case where we do not know that the cohomology ring is Noetherian,
namely the algebras of type \textbf{A5} for $p>2$.

\end{document}